\documentclass[11pt]{article}

\usepackage{etex}
\reserveinserts{28}
\usepackage[all]{xy}

\usepackage{amsfonts}
\usepackage{amsthm}
\usepackage{enumerate}
\usepackage{graphicx}
\usepackage{mathrsfs}
\usepackage{bm}
\usepackage{cite}
\usepackage[colorlinks,linkcolor=blue,citecolor=blue]{hyperref}
\usepackage{amssymb,amsmath} 
\usepackage{hyperref}
\usepackage{makecell}
\usepackage{pgf}
\usepackage{pgffor}
\usepackage{pgfcalendar}
\usepackage{pgfpages}
\usepackage{shuffle,yfonts}
\usepackage{mathtools}
\usepackage{tikz}
\usetikzlibrary{arrows,shapes,chains,patterns}
 \allowdisplaybreaks

\theoremstyle{plain}
\newtheorem{thm}{Theorem}[subsection]
\newtheorem{lem}[thm]{Lemma}
\newtheorem{cor}[thm]{Corollary}
\newtheorem{con}[thm]{Conjecture}
\newtheorem{pro}[thm]{Proposition}

\theoremstyle{definition}

\newtheorem{re}[thm]{Remark}

\newtheorem{defn}{Definition}[section]

\catcode`!=11
\let\!int\int \def\int{\displaystyle\!int}
\let\!lim\lim \def\lim{\displaystyle\!lim}
\let\!sum\sum \def\sum{\displaystyle\!sum}
\let\!sup\sup \def\sup{\displaystyle\!sup}
\let\!inf\inf \def\inf{\displaystyle\!inf}
\let\!cap\cap \def\cap{\displaystyle\!cap}
\let\!max\max \def\max{\displaystyle\!max}
\let\!min\min \def\min{\displaystyle\!min}
\let\!frac\frac \def\frac{\displaystyle\!frac}
\catcode`!=12

\let\oldsection\section
\renewcommand\section{\setcounter{equation}{0}\oldsection}

\allowdisplaybreaks

\DeclareMathOperator*{\Cat}{\mathbf{Cat}}
\DeclareMathOperator*{\sign}{{sign}}

\newcommand{\gd}{\delta}
\newcommand\tpi{{\tilde{\pi}}}
\newcommand\ora{\overrightarrow}
\newcommand\om{{\omega}}
\newcommand\bfs{{\bf s}}
\newcommand\bfk{{\boldsymbol{\sl k}}}
\newcommand\bfl{{\boldsymbol{\sl l}}}
\newcommand\bfn{{\bf n}}
\newcommand\bfp{{\bf p}}
\newcommand\bfq{{\bf q}}
\newcommand\bfB{{\bf B}}
\newcommand\bfTB{{\bf TB}}
\newcommand\bfeta{{\boldsymbol \eta}}
\newcommand\eps{{\varepsilon}}
\newcommand\bfeps{{\boldsymbol \eps}}

\newcommand\CMZV{\mathsf{CMZV}}
\newcommand\AMTV{\mathsf{AMTV}}
\newcommand{\AMMV}{\mathsf {AMMV}}

\newcommand\upi{{i}}

\newcommand\ol{\overline}

\def\R{\mathbb{R}}
\def\N{\mathbb{N}}\def\Z{\mathbb{Z}}
\def\Q{\mathbb{Q}}

\def\su{\sum\limits_{n=1}^\infty}

\def\t{\widetilde{t}}
\def\S{\widetilde{S}}
\def\z{\zeta}
\def\T{{\bar T}}

\def\tt{\left(\frac{1-t}{1+t} \right)}

\newcommand\bfsi{{\boldsymbol \sigma}}

\begin{document}
\title {\bf Alternating Multiple $T$-Values: Weighted Sums, Duality, and Dimension Conjecture}
\author{
{Ce Xu${}^{a,}$\thanks{Email: cexu2020@ahnu.edu.cn}\quad Jianqiang Zhao${}^{b,}$\thanks{Email: zhaoj@ihes.fr}}\\[1mm]
\small a. School of Mathematics and Statistics, Anhui Normal University,\\ \small Wuhu 241000, P.R. China\\
\small b. Department of Mathematics, The Bishop's School, La Jolla,\\ \small  CA 92037, United States of America\\
[5mm]
Dedicated to professor Masanobu Kaneko on the occasion of his 60th birthday}

\date{}
\maketitle \noindent{\bf Abstract} In this paper, we define some weighted sums of the alternating multiple $T$-values (AMTVs), and study several duality formulas for them by using the tools developed in our previous papers. Then we introduce the alternating version of the convoluted $T$-values and Kaneko-Tsumura $\psi$-function, which are proved to be closely related to the AMTVs. At the end of the paper, we study the $\Q$-vector space generated by the AMTVs of any fixed weight $w$ and provide some evidence for the conjecture that their dimensions $\{d_w\}_{w\ge 1}$ form the tribonacci sequence 1, 2, 4, 7, 13, ....




\section{Introduction}

\subsection{Multiple Zeta Values and Several Variations}
Recently, several variants of classical multiple zeta values of level 2 called \emph{multiple $t$-values} (abbr. MtVs), \emph{multiple $T$-values} (abbr. MTVs) and \emph{multiple mixed values} (abbr. MMVs)
were introduced and studied in Hoffman \cite{H2019}, Kaneko--Tsumura \cite{KTA2018,KTA2019} and Xu--Zhao \cite{XZ2020}, defined by
\begin{align*}
&t(k_1,\dotsc,k_r):=\sum_{0<m_1<\cdots<m_r\atop  2\nmid m_j   \ \forall j} \frac{1}{m_1^{k_1}\dotsm m_r^{k_r}},\\
&T(k_1,\dotsc,k_r):= \sum_{0<m_1<\cdots<m_r\atop m_j\equiv j\pmod{2}\ \forall j} \frac{2^r}{m_1^{k_1}\dotsm m_r^{k_r}},
\end{align*}
and
\begin{align*}
M(k_1,\dotsc,k_r;\varepsilon_1,\varepsilon_2,\dotsc,\varepsilon_r):&=\sum_{0<m_1<\cdots<m_r} \frac{(1+\varepsilon_1(-1)^{m_1}) \dotsm (1+\varepsilon_r(-1)^{m_r})}{m_1^{k_1} \cdots m_r^{k_r}} \\
&=\sum_{0<m_1<\cdots<m_r\atop 2| m_j\ \text{if}\ \varepsilon_j=1\ \text{and}\ 2\nmid m_j\ \text{if}\ \varepsilon_j=-1} \frac{2^r}{m_1^{k_1}\dotsm m_r^{k_r}},
\end{align*}
respectively. Here $k_1,\dotsc,k_r$ are positive integers with $k_r>1$ and $\varepsilon_1,\dotsc,\varepsilon_r\in\{\pm 1\}$. It is obvious that MtVs can be regarded as the special case of multiple Hurwitz zeta values. The classical \emph{multiple zeta values} (abbr. MZVs) are defined by (see \cite{H1992,DZ1994})
\begin{align*}
\zeta(k_1,\dotsc,k_r):=\sum\limits_{0<m_1<\dotsb<m_r } \frac{1}{m_1^{k_1}\dotsm m_r^{k_r}},
\end{align*}
for positive integers $k_1,\dotsc,k_r$ with $k_r>1$. We call $k_1+\dotsb+k_r$ and $r$ the \emph{weight} and \emph{depth}, respectively.

The systematic study of MZVs began in the early 1990s with the works of Hoffman \cite{H1992} and Zagier \cite{DZ1994}. Due to their surprising and sometimes mysterious appearance in the study of many branches of mathematics and theoretical physics, these special values have attracted a lot of attention and interest in the past three decades (for example, see the book by the
second author  \cite{Z2016}).
Clearly, MtVs and MTVs are special cases of MMVs. Moreover, it is obvious that MtVs satisfy the series stuffle relation, however, it is highly nontrivial to see that MTVs can be expressed using iterated integral and satisfy the integral shuffle relation (see \cite[Theorem~2.1]{KTA2019}). Further, in \cite{XZ2020}, we found that MMVs satisfy the series stuffle relations and integral shuffle relations.

The subject of this paper is the alternating MTVs (abbr. AMTVs) which are natural level two generalizations of MTVs. We shall derive a few weight sum formulas and a duality formula (see Theorem~\ref{thm:dualityAMTV}). The weight sum formula of MZVs
is widely known and its variants are enormous (see, for e.g., \cite{G2016,G1997,GL2015,H2017,LQ2019,M2013,SC2012} and \cite[Ch. 5]{Z2016}). For Hoffman's MtVs, a number of mathematicians also studied their weighed sum formulas in \cite{LX2020,SJ2017,Z2015}.

In contrast, we know very little about the weighed sum formulas for MTVs. For these values, certain formulas are found only in depths 2 and 3. For details, see Kaneko and Tsumura \cite[Thms. 3.2 and 3.3]{KTA2019}.
Recently, the first author has considered the following weighted sums of AMTVs
\begin{align}
W(k,r):=\sum_{k_1+\dotsb+k_r=k,\atop k_1,\dotsc,k_r\geq 1} T(k_1,\dotsc,k_{r-1},\ol{k_r+m-1}),
\end{align}
and proved two duality formulas for the weighted sums $W(k,r)$ (see \cite[Thms. 1.1 and 1.2]{CX2020}).
Here $T(k_1,\dotsc,k_{r-1}, \ol{k_r})$ stands for a kind of AMTV, which is defined for positive integers $k_1,\dotsc,k_r$ by
\begin{align}\label{S-MTV}
T(k_1,\dotsc,k_{r-1},\ol{k_r}):=2^r\sum\limits_{0<n_1<\cdots<n_r} \frac{(-1)^{n_r}}{(2n_1-1)^{k_1}\dotsm (2n_{r-1}-r+1)^{k_{r-1}}(2n_r-r)^{k_r}}.
\end{align}
Note that $T(k_1,\dotsc,k_{r-1},\ol{k_r})$ is also denoted by $\T(k_1,\dotsc,k_{r-1},k_r)$ in \cite{CX2020}.

More generally, for positive integers $k_1,\dotsc,k_r$ and $\sigma_1,\dotsc,\sigma_r\in \{\pm 1\}$ with $(k_r,\sigma_r)\neq (1,1)$, the general MTV is defined by
\begin{align}\label{Def.MTV}
T(k_1,\dotsc,k_r;\sigma_1,\dotsc,\sigma_r):=2^r\sum\limits_{0<n_1<\cdots<n_r} \frac{\sigma_1^{n_1}\sigma_2^{n_2}\cdots\sigma_r^{n_r}}{(2n_1-1)^{k_1} (2n_2-2)^{k_2}\cdots(2n_r-r)^{k_r}}.
\end{align}
We may compactly indicate the presence of an alternating sign as follows.
Whenever $\sigma_j=-1$,  we place a bar over the corresponding component $k_j$. For example,
\begin{equation*}
T(\bar 2,3,\bar 1,4)=T(2,3,1,4;-1,1,-1,1).
\end{equation*}
If one or more of the $\sigma_j$'s is $-1$ in \eqref{Def.MTV}, then the value is called an \emph{alternating multiple $T$-value} (abbr.  AMTV). In fact, for any  $\bfk=(k_1,k_2,\ldots,k_r)\in \N^r$,
$\bfeps=(\eps_1, \dots, \eps_r)\in\{\pm 1\}^r$, and $\bfsi=(\sigma_1, \dots, \sigma_r)\in\{\pm 1\}^r$ with $(k_r,\sigma_r)\neq (1,1)$, we can define the following more general MMVs
\begin{equation}\label{equ:AMMVdefn}
M_{\bfsi}(\bfk;\bfeps):=\sum_{0<m_1<\cdots<m_r} \frac{(1+\eps_1(-1)^{m_1})\sigma_1^{(2m_1+1-\eps_1)/4} \cdots (1+\eps_r(-1)^{m_r})\sigma_r^{(2m_r+1-\eps_r)/4}}{m_1^{k_1} \cdots m_r^{k_r}}.
\end{equation}
As usual, we call $k_1+\cdots+k_r$ and $r$ the \emph{weight} and \emph{depth}, respectively. If one or more of the $\sigma_j$'s is $-1$ in \eqref{equ:AMMVdefn}, then the value is called an \emph{alternating multiple mixed value} (or \emph{alternating multiple $M$-value}, abbr. AMMV). For convenience, we say the \emph{signature} of $k_j$ is even or odd depending on whether $\eps_j$ is 1 or $-1$. For brevity, we put a check on top of the component $k_j$ if $\eps_j=-1$. For example,
\begin{align*}
M_{\sigma_1,\sigma_2,\sigma_3}(k_1,k_2,\check{k_3})&=8\sum_{0<\ell<m<n} \frac{\sigma_1^\ell\sigma_2^m\sigma_3^n}{(2\ell)^{k_1}  (2m)^{k_2} (2n-1)^{k_3}},\\
M_{\sigma_1,\sigma_2,\sigma_3}(\check{{k_1}},k_2,\check{{k_3}})&=8\sum_{0<\ell<m<n} \frac{\sigma_1^\ell\sigma_2^{m-1}\sigma_3^{n-1}}{(2\ell-1)^{k_1}  (2m-2)^{k_2} (2n-3)^{k_3}}\\&=\sigma_2\sigma_3T(k_1,k_2,k_3;\sigma_1,\sigma_2,\sigma_3),\\
M_{\sigma_1,\sigma_2,\sigma_3}(\check{{k_1}},\check{k_2},\check{{k_3}})&=8\sum_{0<\ell<m<n} \frac{\sigma_1^\ell\sigma_2^{m}\sigma_3^{n}}{(2\ell-1)^{k_1}  (2m-1)^{k_2} (2n-1)^{k_3}}.
\end{align*}

Clearly, up to a sign every AMTV is an AMMV with the starting odd signature (i.e., smallest summation index in \eqref{equ:AMMVdefn} is odd). Further, the AMTVs and AMMVs can be roughly thought as special cases of colored MZVs of level four (the subtle point here is that we need slightly enlarge the set of coefficients), see \cite[Chapter 13-14]{Z2016}. In general, for any $(k_1,\dotsc, k_r)\in\N^r$ and $N$th roots of unity $\eta_1,\dotsc,\eta_r$, we can defined the colored MZVs as
\begin{equation*}
Li_{k_1,\dotsc, k_r}(\eta_1,\dotsc,\eta_r):=\sum\limits_{0<n_1<\cdots<n_r}
\frac{\eta_1^{n_1}\dots \eta_r^{n_r}}{n_1^{k_1} \dots n_r^{k_r}}
\end{equation*}
which converges if $(k_r,\eta_r)\ne (1,1)$ (see \cite[Ch. 15]{Z2016}), in which case we call $(\bfk,\bfeta)$ \emph{admissible}.

The level two colored MZVs are called alternating MZVs. In this case, namely,
when $(\eta_1,\dotsc,\eta_r)\in\{\pm 1\}^r$ and $(k_r,\eta_r)\ne (1,1)$, we set
$\zeta(\bfk;\bfeta)= Li_\bfk(\bfeta)$. Further, we put a bar on top of
$k_j$ if $\eta_j=-1$. For example,
\begin{equation*}
\zeta(\bar 2,3,\bar 1,4)=\zeta(  2,3,1,4;-1,1,-1,1).
\end{equation*}

It is clear that every non-alternating MMV (MtV or MTV) can be written as a $\Q$-linear combination of alternating MZVs.
Further, \eqref{S-MTV} is a special case of \eqref{Def.MTV} with $\sigma_1=\cdots=\sigma_{r-1}=1$ and $\sigma_r=-1$. In particular, if all $\sigma_j=1\ (j=1,2,\dotsc,r)$ in \eqref{Def.MTV}, then it becomes the original (Kaneko--Tsumura) multiple $T$-value, which was introduced and studied by Kaneko and Tsumura in \cite{KTA2018,KTA2019}.

\subsection{Main Results}
The main results of this paper concern some weight sum formula of AMTVs and
some duality relation for AMTVs. First, we extend the tools developed
in our previous papers \cite{CX2020} and \cite{XZ2020} to study the following weight
sum formula of the AMTVs,
\begin{align}\label{a-5}
W_l(k,r):=\sum_{k_1+k_2+\cdots+k_r=k,\atop k_1,\dotsc,k_r\geq 1} T(k_1,\dotsc,k_{l-1},\ol{k_l},k_{l+1},\dotsc,k_{r-1},\ol{k_r}),
\end{align}
where $0\leq l\leq r$. Obviously, $W_0(k,r)=W_r(k,r)=W(k,r)$.
Then we prove the following two theorems concerning the duality of $W_l(k,r)$.
Let $\N$ be the set of positive integers and $\N_0:=\N\cup\{0\}$.

\begin{thm}\label{thm1} For any $m,p\in\N$ and $l\in\N_0$, we have
\begin{equation}\label{r1}
(-1)^m \sum_{j=1}^p \alpha_{p-j} W_l(2j+2m+l-1,2m+l)
=(-1)^p \sum_{j=1}^m \alpha_{m-j} W_l(2j+2p+l-1,2p+l),
\end{equation}
where $\alpha_n=(-1)^n \pi^{2n}/(2n+1)!$.
\end{thm}

\begin{thm}\label{thm2} For any $m,p\in\N$ and $l\in\N_0$,, we have
\begin{equation}\label{r2}
(-1)^m \sum_{j=1}^p \alpha_{p-j}W_l(2j+2m+l-3,2m+l-1) =(-1)^p \sum_{j=1}^m \alpha_{m-j} W_l(2j+2p+l-3,2p+l-1).
\end{equation}
\end{thm}

Clearly, if $l=0$, then the Theorems \ref{thm1} and \ref{thm2} above degenerate to \cite[Theorem~1.1]{CX2020} and \cite[Thm.~1.2]{CX2020}, respectively.

Next, for $\bfk=(k_1,\dotsc,k_r)\in \N^r$ we often write $\bfk_r=\bfk$,
define its weight $|\bfk|:=k_1+k_2+\cdots+k_r$ and adopt the following notation:
\begin{alignat*}{3}
&\ora{\bfk_j}:=(k_1,k_2,\ldots ,k_j),\quad && \overleftarrow{\bfk_j}:=(k_r,k_{r-1},\dotsc,k_{r+1-j}),\\
|&\ora{\bfk_j}|:=k_1+k_2+\cdots+k_j,\quad  & |&\overleftarrow{\bfk_j}|:=k_r+k_{r-1}+\cdots+k_{r+1-j}
\end{alignat*}
with $\ora{\bfk_0}=\overleftarrow{\bfk_0}:=\emptyset$ and $|\ora{\bfk_0}|=|\overleftarrow{\bfk_0}|:=0$.

Another purpose of present paper is to prove the following a more general duality relation for general weighted sums involving AMTVs and binomial coefficients by using the method of iterated integrals.
\begin{thm}\label{thm3} For positive integers $k_1,\dotsc,k_r$ and integers $m,p\geq 0$,
\begin{align}\label{a14}
&\sum_{j_1+j_2+\cdots+j_{|\bfk|+1}=m+1+|\bfk|,\atop j_\ell\geq 1 \forall \ell} \binom{j_{|\bfk|+1}+p-1}{p} \nonumber\\&\quad\quad\quad\quad\times T\left(\Cat_{\substack{l=1}}^r \left\{ \{\ol{j_{|\overleftarrow{\bfk}_{l-1}|+1}}\}\diamond\{j_{|\overleftarrow{\bfk}_{l-1}|+2},,\dotsc,j_{|\overleftarrow{\bfk}_{l}|-1}\}\diamond\{\ol{j_{|\overleftarrow{\bfk}_l|}}\}\right\},\ol{j_{|\bfk|+1}+p} \right)\nonumber\\
&+(-1)^{|\bfk|+1} \sum_{j_1+j_2+\cdots+j_{|\bfk|+1}=p+1+|\bfk|,\atop j_\ell\geq 1 \forall \ell} \binom{j_{|\bfk|+1}+m-1}{m} \nonumber\\&\quad\quad\quad\quad\quad\quad\quad\times T\left(\Cat_{\substack{l=1}}^r \left\{ \{\ol{j_{|\ora{\bfk}_{l-1}|+1}}\}\diamond\{j_{|\ora{k}_{l-1}|+2},\dotsc,j_{|\ora{\bfk}_{l}|-1}\}\diamond\{\ol{j_{|\ora{\bfk}_l|}}\}\right\},\ol{j_{|\bfk|+1}+m} \right)\nonumber\\
&=\sum_{l=1}^r (-1)^{|\overleftarrow\bfk_{r-l}|} \sum_{j=1}^{k_l} (-1)^{j-1} T\left(\{1\}_{m-1},\ol{1},\overleftarrow\bfk_{r-l},\ol{j} \right)T\left(\{1\}_{p-1},\ol{1},\ora\bfk_{l-1},{\ol{k_l-j+1}} \right),
\end{align}
where $(\{1\}_{-1},1):=\emptyset$, $\{l\}_m$ denotes the sequence obtained by repeating $l$ exactly $m$ times,
\begin{align*}
(\ol{a_1}\diamond\{a_2,\dotsc,a_{r-1} \}\diamond \ol{a_r}):= \left\{ {\begin{array}{*{20}{c}} \left(\bar a_1,a_2,\dotsc,a_{r-1},\bar a_r\right),
   {\ \ r\geq 2,}  \\
   \quad\quad\quad\quad (a_1), {\ \ \ \ \ \quad\quad r = 1},  \\
\end{array} } \right.
\end{align*}
and $\Cat_{\substack{l=1}}^r \{\bfs(l)\}$ denotes the concatenated sequence $(\bfs(1),\bfs(2),\dotsc,\bfs(r)).$ If $r=0$, then $ \Cat_{\substack{l=1}}^0 \{\bfs(l)\}:=\emptyset$.
\end{thm}

Setting $k_1=k_2=\cdots=k_r=1$ in Theorem \ref{thm3}, we obtain \cite[Theorem~1.3]{CX2020}. In particular, setting $\bfk=k\in\N$ and $m=0$ in Theorem \ref{thm3} we easily deduce the following corollary.
\begin{cor}\label{cor1.4} For positive integer $k\ge 2$ and integer $p\geq 0$,
\begin{align}\label{a15}
&\sum_{j_1+j_2+\cdots+j_{k+1}=p+k+1,\atop j_\ell\geq 1 \forall \ell} T({\ol {j_1}},j_2,\dotsc,j_{k-1},{\ol {j_k}},{\ol {j_{k+1}}})\nonumber\\
&=(-1)^k T(\bar 1,\{1\}_{k-2},\bar 1, \ol{p+1})+(-1)^k \sum_{j=1}^k (-1)^j T(\bar j)T(\{1\}_{p-1},\bar 1,\ol{k-j+1}),
\end{align}
where $(\{1\}_{-1},\bar 1):=\emptyset$.
\end{cor}

We will prove Theorems \ref{thm1}-\ref{thm3} in the Sections \ref{sec3} and \ref{sec4}. At the end of section~\ref{sec4}, we present Conjecture~\ref{conj:parity}, a parity conjecture on the MMVs which cannot be deduced from the well-known parity principle for multiple polylogarithms proved by Panzer \cite{Panzer2017}.

In the last section, we investigate the structure of the $\Q$-vector space $\AMTV_w$ generated by all AMTVs of weight $w$.
\begin{con} \label{conj:tribonacci} \emph{(=Conjecture~\ref{conj:tribonacci2})}
Let $d_w=\dim_\Q \AMTV_w$ for all $w\ge 1$. Set $d_0=1$. Then
\begin{equation*}
    \sum_{w=0}^\infty d_w t^w = \frac{1}{1-t-t^2-t^3}.
\end{equation*}
Namely, the dimensions form the tribonacci sequence $\{d_w\}_{w\ge 1}=\{1,2,4,7,13,24,\dotsc\}$.
\end{con}

Using Deligne's theorem \cite[Theorem~6.1]{Deligne2010} on colored MVZs of level four (i.e., special values of multiple polylogarithms at fourth roots of unity) we are able to verify Conjecture \ref{conj:tribonacci} rigorously for $w\le 5$ assuming Grothendieck's period conjecture and numerically for $w=6$ by expressing every AMTV in terms of Deligne's basis \eqref{equ:CMZVbasis} using the integer relation detecting algorithm PSLQ \cite{BaileyBr2001}. Here, in order to express every AMTV in terms of Deligne's basis \eqref{equ:CMZVbasis} we need to keep 1500 digits of precision in MAPLE. Furthermore, in every weight $w\le 6$ case we found a basis consisting of elements of the form $T(k_1,\dotsc,k_r;\sigma_1,\dotsc,\sigma_r)$ with a suitable choice of the alternating signs $(\sigma_1,\dotsc,\sigma_r)\in\{\pm 1\}^r$ for each choice of $(k_1,\dots,k_r)\in\{1,2,3\}^r.$ It would be very interesting to determine the exact pattern of this basis in general.

\section{Iterated Integral Expressions}\label{sec2}

In this section, we define a variant of Kaneko--Tsumura {\rm A}-function with $r$-variables, and find some iterated integral expressions of it and AMTVs.

\subsection{Multiple Polylogarithm of Level Two with $r$-Variables}

Kaneko and Tsumura \cite{KTA2018,KTA2019} introduced the following a new kind of multiple polylogarithm functions of level two
\begin{align}\label{b1}
{\rm A}(k_1,\dotsc,k_r;x): &= 2^r\sum\limits_{0 < {n_1} <  \cdots  < {n_r}\atop n_j\equiv j\ {\rm mod}\ 2} {\frac{{{x^{{n_r}}}}}{{n_1^{{k_1}}n_2^{{k_2}} \cdots n_r^{{k_r}}}}}\quad (|x|\leq 1)\nonumber\\
&=2^r \sum\limits_{1 < {m_1} <  \cdots  < {m_r}}  {\frac{{{x^{{2m_r-r}}}}}{{(2m_1-1)^{{k_1}}(2m_2-2)^{{k_2}} \dotsm (2m_r-r)^{{k_r}}}}},
\end{align}
where $k_1,\dotsc,k_r$ are positive integers and $(k_r,x)\neq (1,1)$. We call them Kaneko--Tsumura {\rm A}-functions. Note that $2^{-r}{\rm A}(k_1,\dotsc,k_r;x)$ is denoted by ${\rm Ath}(k_1,\dotsc,k_r;x)$ in \cite{KTA2018}. Clearly, if $k_r>1$ and $x=1$,
then \eqref{b1} are exactly the MTVs. If $x=\mathrm{i}:=\sqrt{-1}$, then
\begin{align*}
{T}(k_1,\dotsc,k_{r-1},\ol{k}_r)=\mathrm{i}^r {\rm A}(k_1,\dotsc,k_r;\mathrm{i}).
\end{align*}
Similarly, we define a variant of \eqref{b1} with $r$-variables by
\begin{align}\label{b2}
&{\rm A}(k_1,\dotsc,k_r;x_1,x_2,\dotsc,x_r)\nonumber\\
&:=2^r \sum\limits_{1 < {m_1} <  \dotsb  < {m_r}}  {\frac{\left(\frac{x_1}{x_2}\right)^{2m_1-1}\left(\frac{x_2}{x_3}\right)^{2m_2-2}\dotsm \left(\frac{x_{r-1}}{x_r}\right)^{2m_{r-1}-r+1}x_r^{{2m_r-r}}}{{(2m_1-1)^{{k_1}}(2m_2-2)^{{k_2}} \dotsm (2m_r-r)^{{k_r}}}}}
\end{align}
for $x_1,\dotsc,x_r\in \mathbb{C}$ with $|x_j|\leq 1$.

Next, we will establish the iterated integral expressions of ${\rm A}(k_1,\dotsc,k_r;x_1,x_2,\dotsc,x_r)$ and of the weighed sums involving AMTVs. By an elementary calculation, we find an iterated integral expression of \eqref{b2} as follows:
\begin{align}\label{b3}
&{\rm A}(k_1,\dotsc,k_r;x_1,x_2,\dotsc,x_r)=\int_0^1 \frac{2x_1dt}{1-x_1^2t^2}\left(\frac{dt}{t}\right)^{k_1-1} \dotsm \frac{2x_{r}\, dt}{1-x_r^2t^2} \left(\frac{dt}{t}\right)^{k_r-1}\nonumber\\
&=2^r\left\{\prod\limits_{j=1}^r\frac{(-1)^{k_j-1}}{\Gamma(k_j)}x_j\right\}\int\nolimits_{D_r} \frac{\log^{k_1-1}\left(\frac{t_1}{t_2}\right)\cdots \log^{k_{r-1}-1}\left(\frac{t_{r-1}}{t_r}\right)\log^{k_r-1}\left(t_r\right)}{(1-x_1^2t_1^2)\dotsm (1-x_{r-1}^2t_{r-1}^2)(1-x_r^2t_r^2)}\, dt_1\cdots dt_r,
\end{align}
where $D_{r}:=\{(t_1,\dotsc,t_r)\mid 0<t_1<\cdots<t_r<1\}\quad (r\in\N)$, and
$$\int_{0}^x f_1(t)dt\, f_2(t)dt\cdots f_k(t)dt:=\int\limits_{0<t_1<\cdots<t_k<x}f_1(t_1)f_2(t_2)\cdots f_k(t_k)\, dt_1dt_2\cdots dt_k.$$

\subsection{Iterated Integral Expressions of AMTVs}
From the integral expression \eqref{b3} of the generalized A-functions, we immediately see that the same shuffle product rule holds for (alternating) MTVs as for (alternating) MZVs which can be found, for e.g., in \cite{Z2016}. To study the shuffle structure among AMTVs, we define the 1-forms
\[\om_{-1}(t):=\frac{2dt}{1+t^2},\quad \om_0(t):=\frac{dt}{t}\quad\text{and}\quad \om_1(t):=\frac{2dt}{1-t^2}.\]
It is not hard to see that for $(k_1,\dotsc,k_r)\in \N^r$ and $(\sigma_1,\dotsc,\sigma_r)\in \{\pm 1\}^r$ with $(k_r,\sigma_r)\neq (1,1)$,
\begin{align}\label{equ:integralAMTV}
\frac{T(k_1,\dotsc,k_r;\sigma_1,\dotsc,\sigma_r)}{\sigma'_1\dotsm\sigma'_r}=
\int_0^1 \om_{\sigma'_1} \om_0^{k_1-1} \dotsm \om_{\sigma'_r} \om_0^{k_r-1}
\end{align}
where $\sigma'_j=\sigma_j\sigma_{j+1}\cdots\sigma_r$ ($1\le j\le r$),
and
\begin{align}\label{bb5}
\frac{T(k_1,\dotsc,k_r;\sigma_1/\sigma_2,\dotsc,\sigma_{r-1}/\sigma_r,\sigma_r)}{\sigma_1\sigma_2\cdots\sigma_r}= \int_0^1 \om_{\sigma_1} \om_0^{k_1-1}\om_{\sigma_2} \om_0^{k_2-1}\dotsm \om_{\sigma_r} \om_0^{k_r-1}.
\end{align}
For future reference, we define two maps $\bfp$ and $\bfq$ by
\begin{align}\label{equ:bfp}
  \bfp\big((k_1,\dotsc,k_r;\sigma_1,\dotsc,\sigma_r)\big):=&\, \om_{\sigma_1'} \om_0^{k_1-1}\om_{\sigma_2'} \om_0^{k_2-1}\dotsm \om_{\sigma_r} \om_0^{k_r-1}, \\
 \bfq\big( \om_{\sigma_1} \om_0^{k_1-1} \dotsm \om_{\sigma_r} \om_0^{k_r-1}\big):=&\,
  (k_1,\dotsc,k_r;\sigma_1/\sigma_2,\dotsc,\sigma_{r-1}/\sigma_r,\sigma_r). \label{equ:bfq}
\end{align}
Hence, AMTVs can be expressed by iterated integrals, which leads to the integral shuffle relations. For example,
\begin{align*}
T(\bar 1)T(\bar 2)=\int_0^1 \om_{-1} \int_0^1 \om_{-1}\om_0=2\int_0^1 \om_{-1}\om_{-1}\om_0+\int_0^1 \om_{-1}\om_0\om_{-1}=2T(1,\bar 2)+T(2,\bar 1),
\end{align*}
and
\begin{align*}
T(\bar 1)T(2)&=-\int_0^1 \om_{-1} \int_0^1 \om_{1}\om_0=-\int_0^1 \om_{-1}\om_{1}\om_0-\int_0^1 \om_{1}\om_{-1}\om_0-\int_0^1 \om_{1}\om_0\om_{-1}\\
&=T(\bar 1,2)+T(\bar 1,\bar 2)+T(\bar 2,\bar 1).
\end{align*}

Similar to the duality relation of MZVs and MTVs, we may use the iterated integral expression of AMTVs
to derive their duality relation.

\begin{thm} \label{thm:dualityAMTV}
For any admissible $\bfk=(k_1,\dotsc,k_r;\sigma_1,\dotsc,\sigma_r)$, we define its dual by
\begin{equation}\label{defn:dualOfk}
\bfk^*:=\bfq\big( \om_1^{k_r-1}u_{\sigma_r}\dotsm \om_1^{k_2-1}u_{\sigma_2\sigma_3\cdots\sigma_r}\om_1^{k_1-1}u_{\sigma_1\sigma_2\sigma_3\cdots\sigma_r}  \big)
\end{equation}
where $u_{-1}:=\om_{-1}$, $u_1:=\om_{0}$, and $\bfq$ is defined by \eqref{equ:bfq}. Then
\begin{equation}\label{equ:DualityAMTV}
T(\bfk)= T(\bfk^*).
\end{equation}
\end{thm}
\begin{proof}
Applying the substitution $t\rightarrow \frac{1-t}{1+t}$ in \eqref{bb5},
we get
\begin{align}\label{equ:duality}
  \int_0^1 \om_{\sigma_1} \om_0^{k_1-1}\om_{\sigma_2} \om_0^{k_2-1}\dotsm \om_{\sigma_r} \om_0^{k_r-1}
= \int_0^1 \om_1^{k_r-1}u_{\sigma_r}\dotsm \om_1^{k_2-1}u_{\sigma_2}\om_1^{k_1-1}u_{\sigma_1}
\end{align}
where $u_{-1}:=\om_{-1}$ and $u_1:=\om_{0}$. By \eqref{equ:integralAMTV} we see that
\begin{align*}
T(k_1,\dotsc,k_r;\sigma_1,\dotsc,\sigma_r)=&\, \left(\prod_{\ell=1}^r \sigma_\ell' \right)   \int_0^1 \om_{\sigma_1\sigma_2\sigma_3\cdots\sigma_r} \om_0^{k_1-1}\om_{\sigma_2\sigma_3\cdots\sigma_r} \om_0^{k_2-1}\dotsm \om_{\sigma_r} \om_0^{k_r-1} \\
=&\, \left( \prod_{\ell=1}^r \sigma_\ell' \right) \int_0^1 \om_1^{k_r-1}u_{\sigma_r} \om_1^{k_{r-1}-1}u_{\sigma_{r-1}\sigma_r}\dotsm \om_1^{k_1-1}u_{\sigma_1\sigma_2\sigma_3\cdots\sigma_r}
\end{align*}
by \eqref{equ:duality}. So by \eqref{bb5} and \eqref{equ:bfq} we have
\begin{equation*}
T(\bfk)= \left( \prod_{\ell=1}^r \sigma_\ell' \right) \left( \prod_{\substack {\ell=1\\ \sigma_\ell'\ne 1}}^r \sigma_\ell' \right)    T(\bfk^*) = \left( \prod_{\ell=1}^r \sigma_\ell' \right)^2 T(\bfk^*)=T(\bfk^*)
\end{equation*}
by the definition \eqref{defn:dualOfk}. This completes the proof of the theorem.
\end{proof}

For example,
\begin{align*}
T(\bar 1,\bar 3,2,\bar 5)&=-\int_0^1 \om_{-1}\om_1\om_0^2\om_{-1}\om_0\om_{-1}\om_0^4\\
&=-\int_0^1 \om_1^4u_{-1}\om_1u_{-1}\om_1^2u_1u_{-1}\\
&=-\int_0^1 \om_1^4\om_{-1}\om_1\om_{-1}\om_1^2\om_{0}\om_{-1}\\
&=T(\{1\}_{3},\{\bar 1\}_4,1,\bar 2,\bar 1).
\end{align*}

Now, we establish some iterated integral identities of weighted sums involving AMTVs. Applying the changes of variables $t_j\rightarrow \frac{1-t_{r+1-j}}{1+t_{r+1-j}}$ to \eqref{b3} gives
\begin{align}\label{b4}
&{\rm A}(k_1,\dotsc,k_{r-1},k_r;x_1,\dotsc,x_r)\nonumber\\
&=4^r\left\{\prod\limits_{j=1}^r\frac{(-1)^{k_j-1}}{\Gamma(k_j)}x_j\right\}\int\nolimits_{D_r} \frac{\log^{k_1-1}\left(\frac{(1-t_r)(1+t_{r-1})}{(1+t_r)(1-t_{r-1})} \right)}{(1+t_r)^2-x_1^2(1-t_r)^2}\cdots  \frac{\log^{k_{r-1}-1}\left(\frac{(1-t_2)(1+t_{1})}{(1+t_2)(1-t_{1})} \right)}{(1+t_2)^2-x_{r-1}^2(1-t_2)^2}\nonumber\\
&\quad\quad\quad\quad\quad\quad\quad\quad\quad\quad\quad\quad \times \frac{\log^{k_{r}-1}\left(\frac{1-t_1}{1+t_1} \right)}{(1+t_1)^2-x_r^2(1-t_1)^2}\, dt_1dt_2\cdots dt_r.
\end{align}
Then replacing $k_r$ by $k_r+m-1$ and summing, by a straightforward calculation we obtain
\begin{align}\label{b5}
&\sum_{k_1+\cdots+k_r=k+r-1,\atop k_1,\dotsc,k_r\geq 1} \binom{k_r+m-2}{m-1} {\rm A}(k_1,\dotsc,k_{r-1},k_r+m-1;x_1,\dotsc,x_r)\nonumber\\
&=\frac{(-1)^{k+m}4^r\, x_1\cdots x_r}{(m-1)!(k-1)!}  \int_{D_r} \frac{\log^{m-1}\left(\frac{1-t_1}{1+t_1} \right)\log^{k-1}\left(\frac{1-t_r}{1+t_r} \right)\, dt_1\dotsm dt_r}{[(1+t_1)^2-x_r^2(1-t_1)^2]\cdots [(1+t_r)^2-x_1^2(1-t_r)^2] }.
\end{align}
On the other hand, according to definition of AMTVs and using \eqref{b2} and \eqref{b3}, by direct calculation we have
\begin{align}\label{b6}
{\rm A}(j_1,j_2,\dotsc,j_{r+l};\{1\}_{l},\{\mathrm{i}\}_r)=\frac{1}{\mathrm{i}^r} T(j_1,\dotsc,j_{l-1},\ol{ j}_l,j_{l+1},\dotsc,j_{l+r-1},\ol{j_{l+r}})
\end{align}
and
\begin{align}\label{b7}
&{\rm A}(j_1,j_2,\dotsc,j_{|\bfk|+1};\mathrm{i},\{1\}_{k_r-1},\mathrm{i},\dotsc,\{1\}_{k_2-1},\mathrm{i},\{1\}_{k_1-1},\mathrm{i})\nonumber\\
&=\frac{1}{\mathrm{i}^{r+1}}  T\left(\Cat_{\substack{l=1}}^r \left\{ \{\ol{j_{|\overleftarrow{\bfk}_{l-1}|+1}}\}\diamond\{j_{|\overleftarrow{\bfk}_{l-1}|+2}, \dotsc,j_{|\overleftarrow{\bfk}_{l}|-1}\}\diamond\{\ol{j_{|\overleftarrow{\bfk}_l|}}\}\right\},\ol{j_{|\bfk|+1}} \right).
\end{align}
Hence, from \eqref{b5}-\eqref{b7} we deduce the following two identities:
\begin{align}\label{b8}
&W_l(k+r+l-1,r+l)\nonumber\\
&=\sum_{j_1+j_2+\cdots+j_{r+l}=k+r+l-1,\atop j_1,j_2,\dotsc,j_{r+p}\geq 1}
    T(j_1,\dotsc,j_{l-1},\ol{j}_l,j_{l+1},\dotsc,j_{l+r-1},\ol{j _{l+r}})\nonumber\\
&=\frac{(-1)^{k+r-1}}{(k-1)!}2^r\int_{D_{r+l}} \frac{\log^{k-1}\left( \frac{1-t_{r+l}}{1+t_{r+l}}\right)}{(1+t_1^2)\dotsm (1+t_r^2)t_{r+1}\cdots t_{r+p}}\, dt_1\cdots dt_{r+p}
\end{align}
and
\begin{align}\label{b9}
&\frac{p!m!(-1)^{p+m+r+1}}{2^{r+1}}\sum_{j_1+j_2+\cdots+j_{|\bfk|+1}=p+1+|\bfk|,\atop  j_\ell\geq 1 \forall \ell} \binom{j_{|\bfk|+1}+m-1}{m} \nonumber\\&\quad\quad\quad\quad\quad\quad\quad\quad\times T\left(\Cat_{\substack{l=1}}^r \left\{ \{\ol{j_{|\overleftarrow{\bfk}_{l-1}|+1}}\}\diamond\{j_{|\overleftarrow{\bfk}_{l-1}|+2},,\dotsc,j_{|\overleftarrow{\bfk}_{l}|-1}\}\diamond\{\ol{j_{|\overleftarrow{\bfk}_l|}}\}\right\},\ol{j_{|\bfk|+1}+m} \right)\nonumber\\
&=\int_0^1 \frac{\log^m\left( \frac{1-t}{1+t}\right)\, dt}{1+t^2} \left( \frac{dt}{t}\right)^{k_1-1} \frac{dt}{1+t^2} \left(\frac{dt}{t}\right)^{k_2-1} \dotsm \frac{dt}{1+t^2} \left( \frac{dt}{t}\right)^{k_r-1} \frac{\log^p\left( \frac{1-t}{1+t}\right)\, dt}{1+t^2}.
\end{align}

\section{Some Special AMTVs and Their Weighted Sums}\label{sec3}

In this section, we give the proofs of Theorems \ref{thm1} and \ref{thm2}, and establish some explicit evaluations for some special AMTVs and weighted sums involving AMTVs. Furthermore, we define a alternating convoluted $T$-values $T(\bfk\bar\circledast\bfl)$ and a alternating version of Kaneko--Tsumura $\psi$-function (called Kaneko--Tsumura $\bar\psi$-function), and study some explicit relations among the alternating convoluted $T$-values, Kaneko--Tsumura $\bar\psi$-values and weighted sums involving AMTVs.

\subsection{Multiple $T$-Harmonic Sums and Multiple $S$-Harmonic Sums}
For ${\bfk}:= (k_1,\dotsc, k_r)\in \N^r$, we put $\bfk_r:=(k_1,\dotsc,k_r).$
For positive integers $m$ and $n$ such that $n\ge m$, we define
\begin{align*}
&D_{n,m} :=
\left\{
  \begin{array}{ll}
\Big\{(n_1,n_2,\dotsc,n_m)\in\N^{m} \mid 0<n_1\leq n_2< n_3\leq \cdots \leq n_{m-1}<n_{m}\leq n \Big\},\phantom{\frac12}\ & \hbox{if $2\nmid m$;} \\
\Big\{(n_1,n_2,\dotsc,n_m)\in\N^{m} \mid 0<n_1\leq n_2< n_3\leq \cdots <n_{m-1}\leq n_{m}<n \Big\},\phantom{\frac12}\ & \hbox{if $2\mid m$,}
  \end{array}
\right.  \\
&E_{n,m} :=
\left\{
  \begin{array}{ll}
\Big\{(n_1,n_2,\dotsc,n_{m})\in\N^{m}\mid 1\leq n_1<n_2\leq n_3< \cdots< n_{m-1}\leq n_{m}< n \Big\},\phantom{\frac12}\ & \hbox{if $2\nmid m$;} \\
\Big\{(n_1,n_2,\dotsc,n_{m})\in\N^{m}\mid 1\leq n_1<n_2\leq n_3< \cdots \leq n_{m-1}< n_{m}\leq n \Big\}, \phantom{\frac12}\ & \hbox{if $2\mid m$.}
  \end{array}
\right.
\end{align*}

\begin{defn} (\cite[Defn. 1.1]{XZ2020}) For positive integer $m$, define
\begin{align}
&T_n({\bfk_{2m-1}}):= \sum_{\bfn\in D_{n,2m-1}} \frac{2^{2m-1}}{(\prod_{j=1}^{m-1} (2n_{2j-1}-1)^{k_{2j-1}}(2n_{2j})^{k_{2j}})(2n_{2m-1}-1)^{k_{2m-1}}},\label{MOT}\\
&T_n({\bfk_{2m}}):= \sum_{\bfn\in D_{n,2m}} \frac{2^{2m}}{\prod_{j=1}^{m} (2n_{2j-1}-1)^{k_{2j-1}}(2n_{2j})^{k_{2j}}},\label{MET}\\
&S_n({\bfk_{2m-1}}):= \sum_{\bfn\in E_{n,2m-1}} \frac{2^{2m-1}}{(\prod_{j=1}^{m-1} (2n_{2j-1})^{k_{2j-1}}(2n_{2j}-1)^{k_{2j}})(2n_{2m-1})^{k_{2m-1}}},\label{MOS}\\
&S_n({\bfk_{2m}}):= \sum_{\bfn\in E_{n,2m}} \frac{2^{2m}}{\prod_{j=1}^{m} (2n_{2j-1})^{k_{2j-1}}(2n_{2j}-1)^{k_{2j}}},\label{MES}
\end{align}
where $T_n({\bfk_{2m-1}}):=0$ if $n<m$, and $T_n({\bfk_{2m}})=S_n({\bfk_{2m-1}})=S_n({\bfk_{2m}}):=0$ if $n\leq m$. Moreover, for convenience sake, we set $T_n(\emptyset)=S_n(\emptyset):=1$. We call \eqref{MOT} and \eqref{MET} are multiple $T$-harmonic sums, and call \eqref{MOS} and \eqref{MES} are multiple $S$-harmonic sums.
\end{defn}

Similar to the definition of convoluted $T$-values $T({\bfk_{r}}{\circledast} {\bfl_{s}})$ (see \cite[Defn. 1.2]{XZ2020}), we use the MTHSs and MSHSs to define the alternating convoluted $T$-values $T({\bfk_{r}}\bar{\circledast} {\bfl_{s}})$, which can be regarded as a alternating $T$-variant of Kaneko--Yamamoto MZVs $\z({\bfk}{\circledast} {\bfl}^\star)$ (for detail, see \cite{KY2018}), where $\bfk\equiv\bfk_r=(k_1,\dotsc,k_r)$ and $\bfl\equiv\bfl_s=(l_1,\dotsc,l_s)$.
\begin{defn} For positive integers $m$ and $p$, the \emph{alternating convoluted $T$-values}
\begin{align}
&T({\bfk_{2m}}\bar{\circledast}{\bfl_{2p}})=2\su \frac{T_n({\bfk_{2m-1}})T_n({\bfl_{2p-1}})}{(2n)^{k_{2m}+l_{2p}}}(-1)^n,\label{Aee}\\
&T({\bfk_{2m-1}}\bar{\circledast}{\bfl_{2p-1}})=2\su \frac{T_n({\bfk_{2m-2}})T_n({\bfl_{2p-2}})}{(2n-1)^{k_{2m-1}+l_{2p-1}}}(-1)^n,\label{Aoo}\\
&T({\bfk_{2m}}\bar{\circledast}{\bfl_{2p-1}})=2\su \frac{T_n({\bfk_{2m-1}})S_n({\bfl_{2p-2}})}{(2n)^{k_{2m}+l_{2p-1}}}(-1)^n,\label{Aeo}\\
&T({\bfk_{2m-1}}\bar{\circledast}{\bfl_{2p}})=2\su \frac{T_n({\bfk_{2m-2}})S_n({\bfl_{2p-1}})}{(2n-1)^{k_{2m-1}+l_{2p}}}(-1)^n\label{Aoe}.
\end{align}
Here we allow $k_{2m}+l_{2p}=1$ in (\ref{Aee}), $k_{2m-1}+l_{2p-1}=1$ in (\ref{Aoo}), $k_{2m}+l_{2p-1}=1$ in (\ref{Aeo}) and $k_{2m-1}+l_{2p}=1$ in (\ref{Aoe}).
\end{defn}

In order to better describe the main results, we need the following two lemmas.
\begin{lem}\label{lem1}
\emph{(cf. \cite[Theorem~2.1]{XZ2020})}
For positive integers $m$ and $n$, the following identities hold
\begin{align}
\int_{0}^1 t^{2n-2} \log^{2m}\tt dt& = \frac{2(2m)!}{2n-1} \sum_{j=0}^m {\bar\z}(2j)T_n(\{1\}_{2m-2j}),\label{ee}\\
\int_{0}^1 t^{2n-2} \log^{2m-1}\tt dt& = -\frac{(2m-1)!}{2n-1}
 \left(2\sum_{j=1}^{m} \ol{ \z}(2j-1)T_n(\{1\}_{2m-2j})+ S_n(\{1\}_{2m-1})\right), \label{eo} \\
\int_{0}^1 t^{2n-1} \log^{2m}\tt dt& =\frac{(2m)!}{2n} \left(2\sum_{j=1}^{m} {\bar\z}(2j-1)T_n(\{1\}_{2m-2j+1})
        +  S_n(\{1\}_{2m}) \right), \label{oe} \\
\int_{0}^1 t^{2n-1} \log^{2m-1}\tt dt& = -\frac{(2m-1)!}{n} \sum_{j=0}^{m-1} {\bar\z}(2j)T_n(\{1\}_{2m-2j-1}), \label{oo}
\end{align}
where ${\bar\z}(k)$ stands for alternating Riemann zeta value denoted by
\begin{align}
{\bar\z}(k):=\su \frac{(-1)^{n-1}}{n^k} \quad \text{and} \quad {\bar\z}(0):=\frac{1}{2}\quad (k\in \N).
\end{align}
\end{lem}

\begin{lem}\label{lemc1} \emph{(cf. \cite[Lemma~5.1]{XZ2020})}
Let $A_{p,q}, B_p, C_p\ (p,q\in \N)$ be any complex sequences. If
\begin{align}\label{e1}
\sum\limits_{j=1}^p A_{j,p}B_j=C_p\quad\text{and}\quad A_{p,p}:=1
\end{align}
hold, then
\begin{align}\label{e2}
B_p=\sum\limits_{j=1}^p C_j \sum\limits_{k=1}^{p-j} (-1)^k \left\{\sum\limits_{i_0<i_1<\cdots<i_{k-1}<i_k,\atop i_0=j,i_k=p} \prod\limits_{l=1}^k A_{i_{l-1},i_l}\right\},
\end{align}
where $\sum\limits_{k=1}^0 (\cdot):=1$.
\end{lem}

\subsection{Several Explicit Formulas of Weighted Sums}
Using the iterated integral identity \eqref{b8} with the help of Lemma \ref{lem1}, we can get the following theorem.
\begin{thm}\label{thm-WST} For any $m,p\in\N$ and $l\in\N_0$,
\begin{align}
&\begin{aligned}\label{b20}
W_l(2p+2m+l-1,2m+l)=2(-1)^m \sum_{j=0}^{p-1} {\bar\z}(2p-2-2j)T\Big(\big(\{1\}_{2m-1},l\big)\bar\circledast\big(\{1\}_{2j+2}\big)\Big),
\end{aligned}\\
&\begin{aligned}\label{b21}
W_l(2p+2m+l-2,2m+l)&=2(-1)^m \sum_{j=0}^{p-2} \ol{ \z}(2p-3-2j)T\Big(\big(\{1\}_{2m-1},l\big)\bar\circledast\big(\{1\}_{2j+2}\big)\Big)\\
&\quad+(-1)^m T\Big(\big(\{1\}_{2m-1},l\big)\bar\circledast\big(\{1\}_{2p-1}\big)\Big),
\end{aligned}\\
&\begin{aligned}\label{b22}
W_l(2p+2m+l-2,2m+l-1)&=2(-1)^{m-1} \sum_{j=0}^{p-1} \ol{ \z}(2p-1-2j)T\Big(\big(\{1\}_{2m-2},l\big)\bar\circledast\big(\{1\}_{2j+1}\big)\Big)\\
&\quad+(-1)^{m-1} T\Big(\big(\{1\}_{2m-2},l\big)\bar\circledast\{1\}_{2p}\Big),
\end{aligned}\\
&\begin{aligned}\label{b23}
W_l(2p+2m+l-3,2m+l-1)=2(-1)^{m-1} \sum_{j=0}^{p-1} \ol{ \z}(2p-2-2j)T\Big(\big(\{1\}_{2m-2},l\big)\bar\circledast\big(\{1\}_{2j+1}\big)\Big).
\end{aligned}
\end{align}
\end{thm}
\begin{proof} Expanding the $(1+t_j)^{-1}$ into geometric series, then the formula \eqref{b8} can be rewritten in the form
\begin{align*}
&W_l(k+r+l-1,r+l)\\
&=\frac{(-1)^{k-1}2^r}{(k-1)!} \sum_{0<m_1<\cdots<m_r} \frac{(-1)^{m_r} \int_0^1 t^{2m_r-r-1}\log^{k-1}\tt dt}{(2m_1-1)(2m_2-2)\dotsm (2m_{r-1}-r+1)(2m_r-r)^l}.
\end{align*}
By straightforward calculations, it is easy to see that if $r=2m$ and $k=2p$, then
\begin{align*}
&W_l(2p+2m+l-1,2m+l)\\
&=\frac{2(-1)^{m-1}}{(2p-1)!} \su (-1)^n \frac{T_n(\{1\}_{2m-1})}{(2n)^l} \int_0^1 t^{2n-1}\log^{2p-1}\tt dt,
\end{align*}
if $r=2m$ and $k=2p-1$, then
\begin{align*}
&W_l(2p+2m+l-2,2m+l)\\
&=\frac{2(-1)^{m}}{(2p-2)!} \su (-1)^n \frac{T_n(\{1\}_{2m-1})}{(2n)^l} \int_0^1 t^{2n-1}\log^{2p-2}\tt dt,
\end{align*}
if $r=2m-1$ and $k=2p$, then
\begin{align*}
&W_l(2p+2m+l-2,2m+l-1)\\
&=\frac{2(-1)^{m}}{(2p-1)!} \su (-1)^n \frac{T_n(\{1\}_{2m-2})}{(2n-1)^l} \int_0^1 t^{2n-2}\log^{2p-1}\tt dt,
\end{align*}
if $r=2m-1$ and $k=2p-1$, then
\begin{align*}
&W_l(2p+2m+l-3,2m+l-1)\\
&=\frac{2(-1)^{m-1}}{(2p-2)!} \su (-1)^n \frac{T_n(\{1\}_{2m-2})}{(2n-1)^l} \int_0^1 t^{2n-2}\log^{2p-2}\tt dt.
\end{align*}
Then, with the help of integrals \eqref{ee}-\eqref{oo}, we may easily deduce these desired formulas.
\end{proof}

Therefore, from \eqref{b20}, \eqref{b23} and Lemma \ref{lemc1}, we can get the following two theorems.
\begin{thm}
For any $m,p\in\N$ and $l\in\N_0$,
\begin{align}\label{c3}
T\Big(\big(\{1\}_{2m-1},l\big)\bar\circledast\big(\{1\}_{2p}\big)\Big)&=2\su \frac{T_n(\{1\}_{2m-1})T_n(\{1\}_{2p-1})}{(2n)^{l+1}}(-1)^n\nonumber\\
&=(-1)^m \sum_{j=1}^p \alpha_{p-j}W_l(2j+2m+l-1,2m+l)
\end{align}
where $\alpha_n=(-1)^n \pi^{2n}/(2n+1)!$ for all $n\in\N_0$.
\end{thm}
\begin{proof} Setting
\begin{align*}
&A_{j,p}=2{\bar\z}(2p-2j),\quad B_j=\su \frac{T_n(\{1\}_{2m-1})T_n(\{1\}_{2j-1})}{(2n)^{l+1}}(-1)^n
\end{align*}
and
\begin{align*}
&C_p=(-1)^m W_l(2p+2m+l-1,2m+l)
\end{align*}
in Lemma \ref{lemc1} and using \eqref{b20}, we see that it suffices to prove $Z(j,p)=\alpha_{p-j}$ where
\begin{equation*}
Z(j,p):=\sum_{k=1}^{p-j}(-2)^k\sum\limits_{i_0<i_1<\cdots<i_{k-1}<i_k,\atop i_0=j,i_k=p} \prod\limits_{l=1}^k \ol{ \z}(2i_l-2i_{l-1}),\quad Z(p,p):=1.
\end{equation*}
We show in \cite[Proposition 5.2]{XZ2020} that for any $j\in\N$ and $n\in\N_0$,
\begin{equation*}
Z(j,j+n)=\frac{(-1)^{n}\pi^{2n}}{(2n+1)!}=\alpha_n.
\end{equation*}
So the theorem follows immediately.
\end{proof}

\begin{thm}
For any $m,p\in\N$ and $l\in\N_0$,
\begin{align}\label{c4}
T\Big(\big(\{1\}_{2m-2},l\big)\bar\circledast\big(\{1\}_{2p-1}\big)\Big)&=2\su \frac{T_n(\{1\}_{2m-2})T_n(\{1\}_{2p-2})}{(2n-1)^{l+1}}(-1)^n\nonumber\\
&=(-1)^{m-1} \sum_{j=1}^p \alpha_{p-j} W_l(2j+2m+l-3,2m+l-1)
\end{align}
where $\alpha_n=(-1)^n \pi^{2n}/(2n+1)!$ for all $n\in\N_0$.
\end{thm}
\begin{proof}  Setting
\begin{align*}
&A_{j,p}:=2{\bar\z}(2p-2j),\quad B_j:=2\su \frac{T_n(\{1\}_{2m-2})T_n(\{1\}_{2j-2})}{(2n-1)^{l+1}}(-1)^n
\end{align*}
and
\begin{align*}
C_p:=(-1)^{m-1}W_l(2p+2m+l-3,2m+l-1)
\end{align*}
in Lemma \ref{lemc1} and using \eqref{b23} yields the desired formula.
\end{proof}

Similarly, applying \eqref{b21}, \eqref{b22}, \eqref{c3} and \eqref{c4}, we know that for positive integers $m$ and $p$, the sums
\[
2\su \frac{T_n(\{1\}_{2m-1})S_n(\{1\}_{2p-2})}{(2n)^{l+1}}(-1)^n\quad \text{and}\quad 2\su \frac{T_n(\{1\}_{2m-2})S_n(\{1\}_{2p-1})}{(2n-1)^{l+1}}(-1)^n
\]
can also be evaluated in terms of products of weighted sums $W_l(k,r)$ and $\alpha_{j-p}$. By elementary calculations, we can get the following theorems.
\begin{thm}
For any $m,p\in\N$ and $l\in\N_0$,
\begin{align*}
&T\Big(\big(\{1\}_{2m-1},l\big)\bar\circledast\big(\{1\}_{2p-1}\big)\Big)
    =2\su \frac{T_n(\{1\}_{2m-1})S_n(\{1\}_{2p-2})}{(2n)^{l+1}}(-1)^n \\
&=(-1)^mW_l(2p+2m+l-2,2m+l)\nonumber\\&\quad-2(-1)^m\sum_{1\leq i\leq j\leq p-1} \alpha_{j-i} {\bar\z}(2p-1-2j) W_l(2i+2m+l-1,2m+l) .
\end{align*}
\end{thm}

\begin{thm}
For any $m,p\in\N$ and $l\in\N_0$,
\begin{align*}
&T\Big(\big(\{1\}_{2m-2},l\big)\bar\circledast\big(\{1\}_{2p}\big)\Big)
    =2\su \frac{T_n(\{1\}_{2m-2})S_n(\{1\}_{2p-1})}{(2n-1)^{l+1}}(-1)^n \\
&=(-1)^{m-1}W_l(2p+2m+l-2,2m+l-1)\nonumber\\&\quad-2(-1)^{m-1}\sum_{1\leq i\leq j\leq p}  \alpha_{j-i}  {\bar\z}(2p+1-2j)W_l(2i+2m+l-3,2m+l-1).
\end{align*}
\end{thm}
\medskip
\noindent
{\bf Proofs of Theorems \ref{thm1} and \ref{thm2}}. Changing $(m,p)$ to $(p,m)$ in \eqref{c3} and \eqref{c4}, and using the duality of series on the left hand sides, we may easily deduce the evaluations \eqref{r1} and \eqref{r2}. Thus, we complete the proofs of Theorems \ref{thm1} and \ref{thm2}.\hfill$\square$

\subsection{Multiple Integrals Associated with 3-Labeled Posets}

In this subsection, we introduce the multiple integrals associated with 3-labeled posets, and define the Kaneko--Tsumura $\bar\psi$-function. We express AMTVs and Kaneko--Tsumura $\bar\psi$-values in terms of
multiple integral associated with 3-labeled posets, which implies that the Kaneko--Tsumura $\bar\psi$-values can be expressed in terms of linear combination of MTVs. Further, we establish some explicit relations between alternating convoluted $T$-values and Kaneko--Tsumura $\bar\psi$-values. The key properties of these integrals was first studied by Yamamoto in \cite{Y2014}.

\begin{defn}
A \emph{$3$-poset} is a pair $(X,\delta_X)$, where $X=(X,\leq)$ is
a finite partially ordered set and $\delta_X$ is a map from $X$ to $\{-1,0,1\}$.
We often omit  $\delta_X$ and simply say ``a 3-poset $X$''.
The $\delta_X$ is called the \emph{label map} of $X$.

Similar to 2-poset, a 3-poset $(X,\delta_X)$ is called \emph{admissible}
if $\delta_X(x) \ne 1$ for all maximal
elements and $\delta_X(x) \ne 0$ for all minimal elements $x \in X$.
\end{defn}

\begin{defn}
For an admissible $3$-poset $X$, we define the associated integral
\begin{equation}
I(X)=\int_{\Delta_X}\prod_{x\in X}\omega_{\delta_X(x)}(t_x),
\end{equation}
where
\[\Delta_X=\bigl\{(t_x)_x\in [0,1]^X \bigm| t_x<t_y \text{ if } x<y\bigr\}\]
and
\[\om_{-1}(t):=\frac{2dt}{1+t^2},\quad \om_0(t):=\frac{dt}{t}\quad\text{and}\quad \om_1(t):=\frac{2dt}{1-t^2}.\]
\end{defn}

For the empty 3-poset, denoted $\emptyset$, we put $I(\emptyset):=1$.

\begin{pro}\label{prop:shuffl3poset}
For non-comparable elements $a$ and $b$ of a $3$-poset $X$, $X^b_a$ denotes the $3$-poset that is obtained from $X$ by adjoining the relation $a<b$. If $X$ is an admissible $3$-poset, then the $3$-poset $X^b_a$ and $X^a_b$ are admissible and
\begin{equation}
I(X)=I(X^b_a)+I(X^a_b).
\end{equation}
\end{pro}

Note that the admissibility of a $3$-poset corresponds to
the convergence of the associated integral. We use Hasse diagrams to indicate $3$-posets, with vertices $\circ$ and ``$\bullet\ \sigma$" corresponding to $\delta(x)=0$ and $\delta(x)=\sigma\ (\sigma\in\{\pm 1\})$, respectively. For convenience, if $\sigma=1$, replace ``$\bullet\ 1$" by $\bullet$ and if $\sigma=-1$, replace ``$\bullet\ -1$" by ``$\bullet\ {\bar1}$".  For example, the diagram
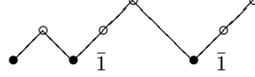
\[\begin{xy}
{(0,-4) \ar @{{*}-o} (4,0)},
{(4,0) \ar @{-{*}} (8,-4)},
{(8,-4) \ar @{-o}_{\bar1} (12,0)},
{(12,0) \ar @{-o} (16,4)},
{(16,4) \ar @{-{*}} (24,-4)},
{(24,-4) \ar @{-o}_{\bar1} (28,0)},
{(28,0) \ar @{-o} (32,4)}
\end{xy} \]
represents the $3$-poset $X=\{x_1,x_2,x_3,x_4,x_5,x_6,x_7,x_8\}$ with order
$x_1<x_2>x_3<x_4<x_5>x_6<x_7<x_8$ and label
$(\delta_X(x_1),\dotsc,\delta_X(x_8))=(1,0,-1,0,0,-1,0,0)$. For composition $\bfk=(k_1,\dotsc,k_r)$ and $\bfsi\in\{\pm 1\}^r$ (admissible or not),
we write
\[\begin{xy}
{(0,-3) \ar @{{*}.o} (0,3)},
{(1,-3) \ar @/_1mm/ @{-} _{(\bfk,\bfsi)} (1,3)}
\end{xy}\]
for the `totally ordered' diagram:
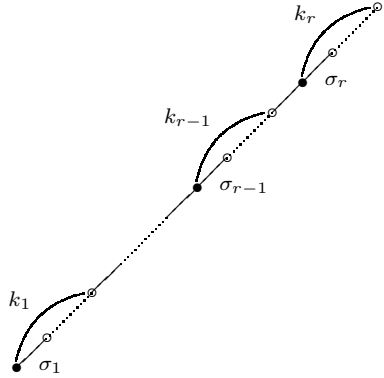
\[\begin{xy}
{(0,-24) \ar @{{*}-o}_{\sigma_1} (4,-20)},
{(4,-20) \ar @{.o} (10,-14)},
{(10,-14) \ar @{-} (14,-10)},
{(14,-10) \ar @{.} (20,-4)},
{(20,-4) \ar @{-{*}} (24,0)},
{(24,0) \ar @{-o}_{\sigma_{r-1}}(28,4)},
{(28,4) \ar @{.o} (34,10)},
{(34,10) \ar @{-{*}} (38,14)},
{(38,14) \ar @{-o}_{\sigma_r} (42,18)},
{(42,18) \ar @{.o} (48,24)},
{(0,-23) \ar @/^2mm/ @{-}^{k_1} (9,-14)},
{(24,1) \ar @/^2mm/ @{-}^{k_{r-1}} (33,10)},
{(38,15) \ar @/^2mm/ @{-}^{k_r} (47,24)}
\end{xy} \]
If $k_i=1$, we understand the notation $\begin{xy}
{(0,-5) \ar @{{*}-o}_{\sigma_i} (4,-1)},
{(4,-1) \ar @{.o} (10,5)},
{(0,-4) \ar @/^2mm/ @{-}^{k_i} (9,5)}
\end{xy}$ as a single $\bullet\ {\sigma_i}$.
We see from \eqref{bb5}
\begin{align}\label{5.19}
I\left(\ \begin{xy}
{(0,-3) \ar @{{*}.o} (0,3)},
{(1,-3) \ar @/_1mm/ @{-} _{(\bfk,\bfsi)} (1,3)}
\end{xy}\right)=\frac{T(k_1,\dotsc,k_r;\sigma_1/\sigma_2,\dotsc,\sigma_{r-1}/\sigma_r,\sigma_r)}{\sigma_1\sigma_2\cdots\sigma_r}.
\end{align}

Recall from \cite{KTA2018,KTA2019} that the Kaneko--Tsumura $\psi$-function is defined by
\begin{align}\label{DefKT-psi}
\psi(k_1,k_2\ldots,k_r;s):=\frac{1}{\Gamma(s)} \int\limits_{0}^\infty \frac{t^{s-1}}{\sinh(t)}{\rm A}({k_1,\dotsc,k_r};\tanh(t/2))\, dt\quad (\Re(s)>0),
\end{align}
where $k_1,\dotsc,k_r$ are positive integers and for $(k_r,x)\neq (1,1)$
\begin{align}\label{DefKT-A}
&{\rm A}(k_1,\dotsc,k_r;x): = 2^r\sum\limits_{1 \le {n_1} <  \cdots  < {n_r}\atop n_i\equiv i\ {\rm mod}\ 2} {\frac{{{x^{{n_r}}}}}{{n_1^{{k_1}}n_2^{{k_2}} \cdots n_r^{{k_r}}}}},\quad x \in \left[ { - 1,1} \right].
\end{align}
Next, we define a alternating version of Kaneko--Tsumura $\psi$-function and we call it Kaneko--Tsumura $\bar\psi$-function.
\begin{defn} For $\bfk=(k_1,\dotsc,k_r)\in\N^r$ and $\Re(s)>0$, we define the Kaneko--Tsumura $\bar\psi$-function by
\begin{align}\label{DefKT-Apsi}
\bar\psi(k_1,k_2\ldots,k_r;s):=\frac{1}{\Gamma(s)} \int\limits_{0}^\infty \frac{t^{s-1}}{\sinh(t)}{\rm B}({k_1,\dotsc,k_r};\tanh(t/2))\, dt,
\end{align}
where for $x \in \left[ { - 1,1} \right]$,
\begin{align*}
{\rm B}(k_1,\dotsc,k_r;x): =
\left\{
\begin{array}{ll} (-1)^m \mathrm{i} {\rm A}(k_1,\dotsc,k_{2m-1};\mathrm{i}x),
   &\quad \hbox{if $r=2m-1$};  \\
   (-1)^m{\rm A}(k_1,\dotsc,k_{2m};\mathrm{i}x), &\quad\hbox{if $r=2m$}.   \\
\end{array}
\right.
\end{align*}
\end{defn}
By an elementary calculation, we obtain the iterated integral
\begin{align*}
{\rm B}(k_1,\dotsc,k_r;x)=\int_0^x \frac{2dt}{1+t^2}\left(\frac{dt}{t}\right)^{k_1-1}
\cdots\frac{2dt}{1+t^2}\left(\frac{dt}{t}\right)^{k_r-1}.
\end{align*}

According to definition and using the fact that if $x=\tanh(t/2)$ and $s=p+1\in\N$ then $dx/x=dt/\sinh(t)$ and $2dx/(1-x^2)=dt$ we deduce that
\begin{align}\label{BPP}
\bar\psi(k_1,k_2\ldots,k_r;p+1)=\frac{(-1)^p}{p!}\int\limits_{0}^1 \frac{\log^p\left(\frac{1-x}{1+x}\right){\rm B}(k_1,\dotsc,k_r;x)}{x}dx.
\end{align}
As an application, we can get the following theorem immediately.
\begin{thm} For $\bfk=(k_1,\dotsc,k_r)\in\N^r$ and integer $p\geq 0$, we have
\begin{equation*}
\bar \psi(\bfk;p+1)=\frac{1}{p!}\,I\left(\xybox{
{(0,-9) \ar @{{*}-o} (0,-4)},
{(0,-4) \ar @{.o} (0,4)},
{(0,4) \ar @{-o} (10,9)},
{(10,9) \ar @{-{*}} (6,-5)},
{(10,9) \ar @{-{*}} (16,-5)},
{(8,-5) \ar @{.} (14,-5)},
{(-1,-9) \ar @/^1mm/ @{-} ^{(\bfk,{\bf\bar1})} (-1,4)},
{(6,-6) \ar @/_1mm/ @{-} _{p} (16,-6)},
}\ \right)=I\left(\xybox{
{(0,-9) \ar @{{*}-o} (0,-4)},
{(0,-4) \ar @{.o} (0,4)},
{(0,4) \ar @{-o} (5,9)},
{(10,-9) \ar @{{*}-{*}} (10,-4)},
{(10,-4) \ar @{.{*}} (10,4)},
{(10,4) \ar @{-} (5,9)},
{(-1,-9) \ar @/^1mm/ @{-} ^{(\bfk,{\bf\bar1})} (-1,4)},
{(11,-9) \ar @/_1mm/ @{-} _{p} (11,4)},
}\ \right)
\end{equation*}
since there are exactly $p!$ ways to impose
a total order on the $p$ black vertices. Here ${\bf\bar1}:=(\{\bar 1\}_r)$.
\end{thm}
By Proposition~\ref{prop:shuffl3poset} this implies the result that $\bar\psi(k_1,\dotsc,k_r;p+1)$ can be expressed as a finite sum of AMTVs. For example,
\begin{align*}
\bar\psi(1,2;2)=T(1,\bar 2,2)+T(1,\bar1,3)+T(\bar1,\bar1,\bar3)+T(\bar1,1,\bar3).
\end{align*}

\begin{thm}\label{thm-BPT}  For positive integers $m$ and $p$,
\begin{align*}
&\begin{aligned}
\bar\psi(\bfk_{2m-1};2p)=(-1)^m2\sum_{j=0}^{p-1}{\bar\z}(2p-1-2j)T(\bfk_{2m-1}\bar\circledast \{1\}_{2j+1})+(-1)^mT(\bfk_{2m-1}\bar\circledast \{1\}_{2p}),
\end{aligned}\\
&\begin{aligned}
\bar\psi(\bfk_{2m-1};2p+1)=(-1)^m2\sum_{j=0}^{p}{\bar\z}(2p-2j)T(\bfk_{2m-1}\bar\circledast \{1\}_{2j+1}),
\end{aligned}\\
&\begin{aligned}
\bar\psi(\bfk_{2m};2p)=(-1)^m2\sum_{j=0}^{p-1}{\bar\z}(2p-2-2j)T(\bfk_{2m}\bar\circledast \{1\}_{2j+2}),
\end{aligned}\\
&\begin{aligned}
\bar\psi(\bfk_{2m};2p+1)=(-1)^m2\sum_{j=0}^{p-1}{\bar\z}(2p-1-2j)T(\bfk_{2m}\bar\circledast \{1\}_{2j+2})+(-1)^mT(\bfk_{2m}\bar\circledast \{1\}_{2p+1}).
\end{aligned}
\end{align*}
\end{thm}
\begin{proof}
This follows immediately from the identity \eqref{BPP} and Lemma \ref{lem1}. We leave the detail to the interested reader.
\end{proof}

Comparing Theorem \ref{thm-WST} and Theorem \ref{thm-BPT} with $k_1=\cdots=k_{r-1}=1$ and $k_r=l$, we obtain the following theorem.
\begin{thm}
For any $l,r\in\N$ and $p\in\N_0$,
\begin{equation*}
\bar\psi(\{1\}_{r-1},l;p+1)=(-1)^r W_l(p+r+l,r+l).
\end{equation*}
\end{thm}

Similar to \cite[Theorem~5.3 and 5.4]{XZ2020}, using Lemma \ref{lemc1} and Theorem \ref{thm-BPT}, we have the following theorems.
\begin{thm}\label{thm-bTZA} For any $\bfk_{2m-1}\in\N^{2m-1}$ and $p\in\N$,
\begin{equation*}
T(\bfk_{2m-1}\bar\circledast \{1\}_{2p+1})=(-1)^m\sum_{j=1}^p \alpha_{p-j}\Big(\bar\psi(\bfk_{2m-1};2j+1)+2\ol{ \zeta}(2j)T(\bfk_{2m-2},\ol{k_{2m-1}+1}) \Big).
\end{equation*}
\end{thm}

\begin{thm}\label{thm-bTZB} For any $\bfk_{2m}\in\N^{2m}$ and $p\in\N$,
\begin{equation*}
T(\bfk_{2m}\bar\circledast \{1\}_{2p})=(-1)^m\sum_{j=1}^p \alpha_{p-j}\bar\psi(\bfk_{2m};2j).
\end{equation*}
\end{thm}

The proofs of Theorems \ref{thm-bTZA} and \ref{thm-bTZB} are completely similar to the proofs of \cite[Theorem~5.3 and 5.4]{XZ2020}) and are thus omitted. Using Theorems \ref{thm-BPT}, \ref{thm-bTZA} and \ref{thm-bTZB},
we can also evaluate explicitly the alternating convoluted $T$-values $T(\bfk_{2m-1}\bar\circledast \{1\}_{2p})$ and $T(\bfk_{2m}\bar\circledast \{1\}_{2p+1})$ in terms of Kaneko--Tsumura $\bar\psi$-values and (alternating) Riemann zeta values. This implies the following result.

\begin{cor}\label{cor:convolutedT}
For $\bfk=(k_1,\ldots,k_r)\in\N^r$ and positive integer $p$, the alternating convoluted $T$-value $T(\bfk\bar\circledast\{1\}_p)$ can be expressed as a $\Z$-linear combination of products of alternating {\rm MTVs} and alternating Riemann zeta values.
\end{cor}
Now, we end this section by the following theorem.
\begin{thm}\label{TV-thm} For any positive integers $l_1,l_2$ and $\bfk_{m}\in \N^m$,
\begin{align*}
T(\bfk_{m}\bar\circledast (l_1,l_2))=f(\bfk_m,l_1,l_2)+(-1)^{[(m+1)/2]}
I\left(
\raisebox{12pt}{\begin{xy}
{(-3,-18) \ar @{{*}-}_{\bar1} (0,-15)},
{(0,-15) \ar @{{o}.} (3,-12)},
{(3,-12) \ar @{{o}.} (9,-6)},
{(9,-6) \ar @{{*}-}_{\bar1} (12,-3)},
{(12,-3) \ar @{{o}.} (15,0)},
{(15,0) \ar @{{o}-} (18,3)},
{(18,3) \ar @{{o}-} (21,6)},
{(21,6) \ar @{{o}.} (24,9)},
{(24,9) \ar @{{o}-} (27,3)},
{(27,3) \ar @{{*}-} (30,6)},
{(30,6) \ar @{{o}.} (33,9)},
{(33,9) \ar @{{o}-} },
{(-3,-17) \ar @/^1mm/ @{-}^{k_1} (2,-12)},
{(9,-5) \ar @/^1mm/ @{-}^{k_{m}} (14,0)},
{(18,4) \ar @/^1mm/ @{-}^{l_2} (23,9)},
{(28,3) \ar @/_1mm/ @{-}_{l_{1}} (33,8)}.
\end{xy}}
\right)
\end{align*}
where $f(\bfk_m,l_1,l_2)=0$ if $m$ is even and $f(\bfk_m,l_1,l_2)=2(-1)^{[(m+1)/2]}\bar\zeta(l_1)T(k_1,\dots,k_{m-1},\overline{k_{m}+l_2})$ if $m$ is odd.
\end{thm}
\begin{proof}
The proof of Theorem \ref{TV-thm} is completely similar to the proof of \cite[Theorem~4.5]{XZ2020} and is thus omitted. We leave the detail to the interested reader.
\end{proof}

\subsection{Some Special Values of AMTVs and Weighted Sums}
Next, we consider some specific cases. Setting $l=m=1$ in \eqref{r2} yields
\begin{align}\label{c17}
T(\bar 1,\{1\}_{2p-2},\bar 1)=(-1)^{p-1} \sum_{j=1}^p  \alpha_{p-j} W_1(2j,2).
\end{align}
According to definition, we see that
\begin{align*}
W_1(2j,2)=\sum_{k_1+k_2=2j,\atop k_1,k_2\geq 1} T(\bar k_1,\bar k_2)
\end{align*}
and
\begin{align*}
T(\bar k_1,\bar k_2)=4\sum_{0<m<n} \frac{(-1)^{m+n}}{(2m-1)^{k_1}(2n-2)^{k_2}}=-\frac{1}{2^{k_1+k_2-2}}\sum_{n=1}^\infty \frac{\ol{ h}^{(k_1)}_n}{n^{k_2}}(-1)^{n-1},
\end{align*}
where $\ol{h}_n^{(p)}$ stands for alternating odd harmonic number of order $p$ defined by
\begin{align*}
{ \bar h}_n^{(p)}: = \sum\limits_{k=1}^n {\frac{(-1)^{k-1}}{{{(k-1/2)^p}}}},\quad \ol{h}_n\equiv\ol{h}^{(1)}_n,\quad \ol{ h}_0^{(p)}:=0.
\end{align*}
Note that the first author and Wang prove the following result in \cite[Corollary~3.4]{WX2020}:
\begin{align}\label{c18}
&(1+(-1)^{p+q})\su \frac{\ol{h}^{(p)}_n}{n^q}(-1)^{n-1}\nonumber\\
=&(-1)^p(1+(-1)^q)\ol{t}(p){\bar\z}(q)-(-1)^p\binom{p+q-1}{p-1}\ol{t}(p+q)\nonumber\\
&-(-1)^p\sum_{k=0}^{p-1} ((-1)^k-1)\binom{p+q-k-2}{q-1}{\t}(k+1)\ol{t}(p+q-k-1)\nonumber\\
&+2(-1)^p \sum_{j=1}^{[q/2]} \binom{p+q-2j-1}{p-1}{\z}(2j)\ol{t}(p+q-2j),
\end{align}
where $\t(p):=2^p t(p)=(2^p-1)\z(p)$ for $p>1$, and $\ol{t}(k)$ is the alternating $t$-value defined by
\begin{equation*}
\ol{t}(p):=\su \frac{(-1)^{n-1}}{(n-1/2)^p}\quad (p\in\N).
\end{equation*}
In particular, $\ol{t}(2)=4G$ where $G:=\sum_{n=1}^\infty (-1)^{n-1}/(2n-1)^2$ is Catalan's constant, and
$\ol{t}(2k+1)$ is related to the Euler number $E_{2k}$ by
\begin{equation*}
\ol{t}(2k+1)=\frac{(-1)^kE_{2k}\pi^{2k+1}}{2(2k)!}\quad (k\geq 0), \quad \text{where}\quad
\sec(x)=\sum\limits_{k=0}^{\infty} \frac{(-1)^kE_{2k}}{(2k)!}x^{2k}.
\end{equation*}
In particular, we have $E_0=1,E_2=-1,E_4=5,E_6=-61$ and $E_8=1385$.

Clearly,
\begin{equation*}
T(\bar k)=-\frac1{2^{k-1}}\ol{t}(k)\quad (k\in\N).
\end{equation*}
Therefore, applying \eqref{c18} we obtain
\begin{align}\label{c20}
W_1(2j,2)=-\frac{1}{2^{2j-2}}\sum_{k=1}^j \ol{t}(2k){\bar\z}(2j-2k).
\end{align}
Taking $j=1,2$ and $3$ we see that
\begin{align*}
W_1(2,2)=-\frac{1}{2}\ol{t}(2),\quad W_1(4,2)=-\frac1{8}\ol{t}(4)-\frac1{8}\ol{t}(2)\z(2)
\end{align*}
and
\begin{align*}
W_1(6,2)=-\frac{7}{128}\ol{t}(2)\z(4)-\frac{1}{32}\ol{t}(4)\z(2)-\frac1{32}\ol{t}(6).
\end{align*}
Further, plugging \eqref{c20} into \eqref{c17}, we can get the following theorem.
\begin{thm}\label{thmc6}
For any $p\in\N$, we have
\begin{equation}\label{c21}
T(\bar 1,\{1\}_{2p-2},\bar 1)
=(-1)^{p} \sum_{1\leq k\leq j\leq p} \frac{\ol{t}(2k)\ol{ \z}(2j-2k)\alpha_{j-p}}{2^{2j-2}}
=T(\{1\}_{2p-2},\bar 2).
\end{equation}
\end{thm}
\begin{proof}
The first equality follows immediately from \eqref{c17} and \eqref{c20}.
The second equality follows from the duality relation $
T(\{1\}_{p-1},\ol{1},\{1\}_{r-1},\ol{1})=T(\{1\}_{r-1},\ol{p+1})$ (see \cite[(4.6)]{CX2020}).
\end{proof}

Setting $p=2$ and $3$ we obtain
\begin{align*}
T(\bar 1,1,1,\bar 1)&\, =\frac1{8}\ol{t}(4)-\frac{3}{8}\ol{t}(2)\z(2)=T(1,1,\bar 2)\\
T(\bar 1,1,1,1,1,\bar 1)&\, =-\frac{15}{128}\ol{t}(2)\z(4)+\frac{3}{32}\ol{t}(4)\z(2)-\frac1{32}\ol{t}(6)=T(1,1,1,1,\bar 2).
\end{align*}
Moreover, in \cite[Theorem~2.3 and (4.10)]{CX2020}, the first author shows that
\begin{align}\label{c22}
T(\{1\}_{r-1},\bar k)=\sum_{j=1}^r (-1)^{j-1} T(\{1\}_{r-j-1},\bar 1)W(k+j-1,j),\\
\label{bc1}
W(k+r-1,r)=\sum_{j=1}^r (-1)^{j-1} T(\{1\}_{r-j-1},\bar 1)T(\{1\}_{j-1},\bar k),
\end{align}
where $k,r$ are positive integers and $(\{1\}_{-1},\ol{1}):=\emptyset$.
If we set $r=2p-1$ and $k=2$ in \eqref{c22}, and $r=2p$ and $k=2$  in \eqref{bc1}, then we get
\begin{align}\label{c23}
T(\{1\}_{2p-2},\bar 2)&=\sum_{j=1}^{2p-1} (-1)^{j-1}T(\{1\}_{2p-2-j},\bar 1) W(j+1,j)\nonumber\\
&=\sum_{j=1}^{p} T(\{1\}_{2p-2j-1},\bar 1) W(2j,2j-1) -\sum_{j=1}^{p-1} T(\{1\}_{2p-2j-2},\bar 1) W(2j+1,2j)
\end{align}
and
\begin{align}\label{bc2}
W(2p+1,2p)&=\sum_{j=1}^{2p} (-1)^{j-1} T(\{1\}_{2p-j-1},\bar 1)T(\{1\}_{j-1},\bar 2)\nonumber\\
&=\sum_{j=1}^{p}  T(\{1\}_{2p-2j},\bar 1)T(\{1\}_{2j-2},\bar 2) -\sum_{j=1}^{p} T(\{1\}_{2p-2j-1},\bar 1)T(\{1\}_{2j-1},\bar 2).
\end{align}
Furthermore, the first author proves in \cite[Corollary~2.5]{CX2020} that for any $p\in\N$, the weighted sums $W(2p+1,2p)$ can be expressed explicitly in terms of (alternating) Riemann zeta values by providing an explicit formula. Hence, applying \eqref{c21}, \eqref{c23} and the identity (see \cite[(2.15)]{CX2020})
\begin{align*}
{T}(\{1\}_{r-1},\bar 1)=\frac{(-1)^r}{r!}\left(\frac{\pi}{2}\right)^r,
\end{align*}
we arrive at the following result.
\begin{thm}\label{thmc7}
For any $p\in\N$, the two AMTVs $T(\bar 1,\{1\}_{2p-1},\bar 1)$ and $T(\{1\}_{2p-1},\ol{ 2})$, and the weighted sums $W(2p,2p-1)$ can be expressed as a $\Q$-linear combination of
products of the alternating $t$-values and the Riemann zeta values.
\end{thm}
For example, we compute the following cases
\begin{align*}
&T(1,\bar 2)=-\frac{7}{4}\z(3)+\frac{\pi}{4}\ol{t}(2)=T(\bar 1,1,\bar 1),\\
&T(1,1,1,\bar 2)=\frac{31}{16}\z(5)-\frac{1}{16}\pi\ol{t}(4)-\frac{7}{96}\pi^3\ol{t}(2)=T(\bar 1,1,1,1,\bar 1),
\end{align*}
and
\begin{align*}
&W(4,3)=\frac1{8}\ol{t}(4)+\ol{t}(2)\z(2)-\frac7{8} \pi \z(3),\\
&W(6,5)=-\frac 1{32}\ol{t}(6)-\frac{7}{192}\pi^3\z(3)-\frac{15}{8}\ol{t}(2)\z(4)+\frac{31}{32}\pi \z(5).
\end{align*}
Hence, from Theorems \ref{thmc6} and \ref{thmc7}, we can conclude that for any positive integer $p$,
\begin{align*}
T(\bar 1,\{1\}_{p-1},\bar 1)=T(\{1\}_{p-1},\ol{2})\in \mathbb{Q}[\ol{t}(1),\z(2),\ol{t}(2),\z(3),\ol{t}(3),\z(4),\ldots].
\end{align*}
Further, setting $r=2p-1$ and $k=2$ in \eqref{bc1} we find
\begin{align}
W(2p,2p-1)=\sum_{j=1}^{2p-1} (-1)^{j-1} T(\{1\}_{2p-j-2},\bar 1)T(\{1\}_{j-1},\bar 2).
\end{align}

Very recently, the first author and Wang prove in \cite[Theorem~4.9]{WWXC2020} that for any positive integers $k_1,k_2,k_3$ and $(\sigma_1,\sigma_2,\sigma_3)\in \{\pm1\}^3$ with $(k_2,\sigma_2)\neq (1,1)$ and $(k_3,\sigma_3)\neq (1,1)$, the Kaneko--Tsumura triple $T$-values \[(1+\sigma_1\sigma_3(-1)^{k_1+k_2+k_3})T(k_1,k_2,k_3;\sigma_1,\sigma_2,\sigma_3)\]
can be expressed in terms of combinations of (alternating) double $M$-values and single $M$-values, and give explicit though very complicated formula. Hence, from Theorem \ref{thm2} we can obtain the following theorem.
\begin{thm}\label{thm-WTV} For positive integer $p$, the two AMTVs
\begin{align}
T(1,\bar 1,\{1\}_{2p-2},\bar 1) \quad\text{and}\quad T(\{1\}_{2p-2},\bar 3)
\end{align}
can be expressed as a $\Q$-linear combinations of products of (alternating) double $M$-values and single $M$-values.
\end{thm}
\begin{proof}
Setting $m=1$ and $l=2$ in Theorem \ref{thm2} yields
\[W_2(2p+1,2p+1)=T(1,\bar 1,\{1\}_{2p-2},\bar 1)=(-1)^{p-1}\sum_{j=1}^p \alpha_{p-j}W_2(2j+1,3).\]
From \cite[Theorem~4.9]{WWXC2020} with $\sigma_1=1,\sigma_2=\sigma_3=-1$, we know that the triple $T$-values $T(k_1,\bar k_2,\bar k_3)$ are reducible to combinations of double $M$-values and single $M$-values. Hence, using the definition (\ref{a-5}), we have
\[W_2(2k+1,3)=\sum_{k_1+k_2+k_3=2k+1,\atop k_1,k_2,k_3\geq 1} T(k_1,\bar k_2,\bar k_3).\]
Then, using the well-known duality relation (see \cite[(4.6)]{CX2020})
\[T(\{1\}_{p-1},\ol{1},\{1\}_{r-1},\ol{1})=T(\{1\}_{r-1},\ol{p+1})\]
we obtain the desired description.
\end{proof}

For example, we calculate the following cases:
\begin{align*}
W_2(3,3)&=-\frac{\pi^3}{16},\\
W_2(5,3)&=-\frac{241\pi^5}{11520}-\frac{\pi^3}{24}\log^3(2)+\frac{\pi}{24}\log^4(2)+\pi {\rm Li}_4(1/2)+\frac{\pi}{8}\log(2)\z(3),\\
W_2(7,3)&=-\frac{47\pi^7}{11520}-\frac{\pi^5}{192}\log^2(2)+\frac{\pi^3}{192}\log^4(2)+\frac{\pi^3}{8}{\rm Li}_4(1/2)+\frac{7\pi^3}{64}\log(2)\z(3),\\
&\quad+\frac{\pi}{2} \left(T(1,\bar 5)+T(2,\bar 4)+T(3,\bar 3)+T(4,\bar 2)+T(5,\bar 1) \right),
\end{align*}
and
\begin{align*}
&T(1,\bar 1,\bar 1)=T(\bar 3)=-\frac{\pi^3}{16},\\
&T(1,\bar 1,1,1,\bar 1)=T(1,1,\bar 3)=\frac{121\pi^5}{11520}+\frac{\pi^3}{24}\log^2(2)-\frac{\pi}{24}\log^4(2)-\pi {\rm Li}_4(1/2)-\frac{7\pi}{8}\log(2)\z(3),\\
&T(1,\bar 1,\{1\}_4,\bar 1)=T(\{1\}_4,\bar 3)=-\frac{77\pi^7}{69120}+\frac{\pi^5}{576}\log^2(2)-\frac{\pi^3}{576}\log^4(2)-\frac{\pi^3}{24}{\rm Li}_4(1/2)\\
&\quad\quad\quad\quad\quad\quad\quad\quad\quad\quad-\frac{7\pi^3}{192}\log(2)\z(3)+\frac{\pi}{2} \left(T(1,\bar 5)+T(2,\bar 4)+T(3,\bar 3)+T(4,\bar 2)+T(5,\bar 1) \right).
\end{align*}

\section{Special Values of AMTVs and Several Duality Formulas}\label{sec4}
In this section we prove the duality formula in Theorem \ref{thm3} and find some duality relations of AMTVs by using the method of iterated integrals.

\subsection{Proof of A Duality Formula of Weighted Sums}
To prove Theorem \ref{thm3}, we need the following two lemmas. The first one follows quickly from the general theory of Chen's iterated integrals.
\begin{lem}\label{lem4.1} \emph{(cf. \cite[(1.6.1-2)]{KTChen1971})}
If $f_i\ (i=1,\dotsc,m)$ are integrable real functions, the following identity holds:
\begin{equation*}
 g\left( {{f_1},\dotsc,{f_m}} \right) + {\left( { - 1} \right)^m}g\left( {{f_m},\dotsc,{f_1}} \right)
 = \sum\limits_{i = 1}^{m - 1} {{{\left( { - 1} \right)}^{i - 1}}g\left( {{f_i},{f_{i - 1}}, \cdots ,{f_1}} \right)} g\left( {{f_{i + 1}},{f_{i + 2}} \cdots ,{f_m}} \right),
\end{equation*}
where $g\left( {{f_1}, \dotsc,{f_m}} \right)$ is defined by
\[g\left( {{f_1},\dotsc,{f_m}} \right): = \int\limits_{0 < {t_m} <  \cdots < {t_1} < 1}
{f_1}\left( {{t_1}} \right) \cdots {f_m}\left( {{t_m}} \right)d{t_1}  \cdots d{t_m} .\]
\end{lem}

\begin{lem} For any $\bfk=(k_1,\dotsc,k_r)\in \N^r$ and integer $p\geq 0$,
\begin{align}\label{d2}
&\int_0^1 \frac{\log^p\left( \frac{1-t}{1+t}\right)\, dt}{1+t^2} \left( \frac{dt}{t}\right)^{k_1-1} \frac{dt}{1+t^2} \left( \frac{dt}{t}\right)^{k_2-1} \dotsm \frac{dt}{1+t^2} \left(\frac{dt}{t}\right)^{k_r-1}\nonumber\\
&=\frac{(-1)^{p+r}p!}{2^r} T(\{1\}_{p-1},\ol{1},k_1,\dotsc,k_{r-1},\ol{k_r}),
\end{align}
where $(\{1\}_{-1},1):=\emptyset$.
\end{lem}
\begin{proof} First, we note the fact that ${\rm A}(\{1\}_p;x)=\frac{1}{p!} \left({\rm A}(1;x)\right)^r=\frac{(-1)^p}{p!}\log^p\left(\frac{1-x}{1+x} \right)$ (see \cite[Lemma~5.1 (ii)]{KTA2018} or \cite[(2.14)]{CX2020}). Then
according to definition, we have
\begin{align}\label{d3}
\log^p\left(\frac{1-x}{1+x} \right)=(-1)^p p! 2^p \sum\limits_{1 < {m_1} <  \cdots  < {m_r}}  {\frac{{{x^{{2m_p-p}}}}}{{(2m_1-1)(2m_2-2) \dotsm (2m_p-p)}}}.
\end{align}
Now we can apply \eqref{d3} to the integral of \eqref{d2} and expand the $(1+t_j)^{-1}$ into geometric series to deduce the desired formula with an elementary calculation.
\end{proof}

Next, we prove a duality formula about the iterated integral of \eqref{b9}.
\begin{thm} \label{thm:d4}
For any $\bfk=(k_1,\dotsc,k_r)\in \N^r$ and integers $p,m\geq 0$,
\begin{align*}
&\int_0^1 \frac{\log^p\left( \frac{1-t}{1+t}\right)\, dt}{1+t^2} \left(\frac{dt}{t}\right)^{k_1-1} \frac{dt}{1+t^2} \left(\frac{dt}{t}\right)^{k_2-1} \dotsm \frac{dt}{1+t^2} \left(\frac{dt}{t}\right)^{k_r-1} \frac{\log^m\left( \frac{1-t}{1+t}\right)\, dt}{1+t^2}\nonumber\\
&+(-1)^{|\bf k|+1}\int_0^1 \frac{\log^m\left( \frac{1-t}{1+t}\right)\, dt}{1+t^2} \left(\frac{dt}{t}\right)^{k_r-1} \frac{dt}{1+t^2} \left(\frac{dt}{t}\right)^{k_{r-1}-1} \dotsm \frac{dt}{1+t^2} \left(\frac{dt}{t}\right)^{k_1-1} \frac{\log^p\left( \frac{1-t}{1+t}\right)\, dt}{1+t^2}\nonumber\\
&=C_{m,p,r}\sum_{l=1}^r (-1)^{|\overleftarrow\bfk_{r-l}|} \sum_{j=1}^{k_l} (-1)^{j-1} T\left(\{1\}_{m-1},\ol{ 1},\overleftarrow\bfk_{r-l},\ol{j} \right)T\left(\{1\}_{p-1},\ol{1},\ora\bfk_{l-1},{\ol{k_l-j+1}} \right),
\end{align*}
where $C_{m,p,r}:=\frac{(-1)^{m+p+r+1}m!p!}{2^{r+1}}$.
\end{thm}
\begin{proof}
Let $J_p(k_1,\dotsc,k_r)$ be the integral on the left hand of \eqref{d2}. Using Lemma \ref{lem4.1}, we deduce
\begin{align*}
&\int_0^1 \frac{\log^p\left( \frac{1-t}{1+t}\right)\, dt}{1+t^2} \left(\frac{dt}{t}\right)^{k_1-1} \frac{dt}{1+t^2} \left(\frac{dt}{t}\right)^{k_2-1} \dotsm \frac{dt}{1+t^2} \left(\frac{dt}{t}\right)^{k_r-1} \frac{\log^m\left( \frac{1-t}{1+t}\right)\, dt}{1+t^2}\\
&=\sum_{l=1}^{k_r} (-1)^{l-1} J_m(l)J_p(k_1,\dotsc,k_{r-1},k_r-l+1)\\
&\quad+(-1)^{k_r} \sum_{l=1}^{k_{r-1}} (-1)^{l-1} J_m(k_r,l)J_p(k_1,\dotsc,k_{r-2},k_{r-1}-l+1)\\
&\quad+(-1)^{k_r+{k_{r-1}}} \sum_{l=1}^{k_{r-2}} (-1)^{l-1} J_m(k_r,k_{r-1},l)J_p(k_1,\dotsc,k_{r-3},k_{r-2}-l+1)\\
&\quad+\cdots\\
&\quad+(-1)^{k_r+{k_{r-1}}+\cdots+k_2} \sum_{l=1}^{k_{1}} (-1)^{l-1} J_m(k_r,k_{r-1},\dotsc,k_2,l)J_p(k_1-l+1)\\
&\quad+(-1)^{|\bf k|}\int_0^1 \frac{\log^m\left( \frac{1-t}{1+t}\right)\, dt}{1+t^2} \left(\frac{dt}{t}\right)^{k_r-1} \frac{dt}{1+t^2} \left(\frac{dt}{t}\right)^{k_{r-1}-1} \dotsm \frac{dt}{1+t^2} \left(\frac{dt}{t}\right)^{k_1-1} \frac{\log^p\left( \frac{1-t}{1+t}\right)\, dt}{1+t^2}.
\end{align*}
We can now complete the proof by applying the identity
\begin{align*}
J_p(k_1,\dotsc,k_r)=\frac{(-1)^{p+r}p!}{2^r} T(\{1\}_{p-1},\ol{1},k_1,\dotsc,k_{r-1},\ol{k_r})
\end{align*}
with a straightforward calculation,
\end{proof}

\medskip
\noindent
{\bf Proof of Theorem \ref{thm3}}. Plugging \eqref{b9} into Theorem~\ref{thm:d4} we get the desired result \eqref{a14}.\hfill$\square$

\bigskip
Further, using \eqref{b5} and Lemma \ref{lem4.1}, we can also get the following theorem.

\begin{thm}\label{thmd4} For positive integers $m,k$ and $r$,
\begin{align}
&\sum_{k_1+\cdots+k_r=k+r-1,\atop k_1,\dotsc,k_r\geq 1} \binom{k_r+m-2}{m-1}{\rm A}(k_1,\dotsc,k_{r-1},k_r+m-1;x_1,\dotsc,x_{r-1},x_r)\nonumber\\
&+(-1)^r \sum_{k_1+\cdots+k_r=m+r-1,\atop k_1,\dotsc,k_r\geq 1} \binom{k_r+k-2}{k-1}{\rm A}(k_1,\dotsc,k_{r-1},k_r+k-1;x_r,\dotsc,x_{2},x_1)\nonumber\\
&=\sum_{j=1}^{r-1} (-1)^{j-1} {\rm A}(\{1\}_{r-1-j},m;x_{j+1},x_{j+2},\dotsc,x_r){\rm A}(\{1\}_{j-1},k;x_j,x_{j-1},\dotsc,x_1).
\end{align}
\end{thm}
\begin{proof} This follows immediately from the \eqref{b5} and Lemma \ref{lem4.1}. We leave the detail to the
interested reader.
\end{proof}

In particular, if $r=2$ then we obtain the decomposition
\begin{align*}
{\rm A}(k;x_1){\rm A}(m;x_2)=& \sum_{k_1+k_2=k+1,\atop k_1,k_2\geq 1} \binom{k_2+m-2}{m-1} {\rm A}(k_1,k_2+m-1;x_1,x_2)\\
+&  \sum_{k_1+k_2=m+1,\atop k_1,k_2\geq 1} \binom{k_2+k-2}{k-1} {\rm A}(k_1,k_2+k-1;x_2,x_1).
\end{align*}

It is clear that we can find a lot of duality relations of (alternating) MTVs from Theorem \ref{thmd4}. For example, setting $x_1=x_2=\cdots=x_r=1$ and $m,k\geq 2$ yields the well-known result (see \cite[Theorem~5.7]{KTA2018})
\begin{multline*}
\sum_{j=1}^{r-1} (-1)^{j-1} T(\{1\}_{r-1-j},m)T(\{1\}_{j-1},k)=
\sum_{|\bfk_r|=k+r-1} \binom{k_r+m-2}{m-1}T(\bfk_{r-1},k_r+m-1) \\
+(-1)^r \sum_{|\bfk_r|=m+r-1} \binom{k_r+k-2}{k-1}T(\bfk_{r-1},k_r+k-1),
\end{multline*}
where $\bfk_r=(k_1,\dotsc,k_r)\in\N^r$ as before.
Similarly, setting $x_1=x_2=\cdots=x_r=-1$ we recover the following formula (see \cite[(4.9)]{CX2020})
\begin{multline*}
\sum_{j=1}^{r-1} (-1)^{j-1} T(\{1\}_{r-1-j},\bar m)T(\{1\}_{j-1},\bar k)=\sum_{|\bfk_r|=k+r-1} \binom{k_r+m-2}{m-1}T(\bfk_{r-1},\ol{k_r+m-1}) \\
 +(-1)^r \sum_{|\bfk_r|=m+r-1} \binom{k_r+k-2}{k-1}T(\bfk_{r-1},\ol{k_r+k-1}).
\end{multline*}

\subsection{Several Duality Formulas of AMTVs}
We can also find many interesting results about AMTVs from Theorem \ref{thm3}. For example,
setting $m=p=0$ in \eqref{a14} yields the following corollary.
\begin{cor}\label{cord4} For any $\bfk=(k_1,\dotsc,k_r)\in \N^r$, we have
\begin{align}\label{d5}
&T\left(\Cat_{\substack{l=1}}^r \{\ol{1}\diamond\{1\}_{k_{r+1-l}-2}\diamond\ol{1}\},\bar 1\right)+(-1)^{|\bfk|+1}T\left(\Cat_{\substack{l=1}}^r \{\ol{1}\diamond\{1\}_{k_{l}-2}\diamond\ol{1}\},\bar 1\right)\nonumber\\
&=\sum_{l=1}^r (-1)^{|\overleftarrow{\bfk}_{r-l}|} \sum_{j=1}^{k_l} (-1)^{j-1} T\left(\overleftarrow\bfk_{r-l},\ol{j} \right)T\left(\ora\bfk_{l-1},{\ol{k_l-j+1}} \right).
\end{align}
\end{cor}
In particular, if putting $r=1$ and $k_1=2p+1\ (p\geq 0)$,  then we can rewrite\eqref{d5} in the form
\begin{align}\label{d6}
T(\bar 1\diamond\{1\}_{2p-1}\diamond \bar 1,\bar 1)=\frac{(-1)^p}{2}T^2(\ol{p+1})+\sum_{j=1}^p (-1)^{j-1} T(\bar j)T(\ol{2p+2-j}).
\end{align}
Further, setting $p=0$ and $1$ we get
\begin{equation*}
T(1,\bar 1)=\frac{3}{4}\z(2), \qquad
T(\bar 1,1,\bar 1,\bar 1)=-\frac1{8} \ol{t}^2(2)+\frac{45}{16}\z(4).
\end{equation*}

We now turn to a duality type theorem of AMTVs.
\begin{thm}\label{thmd5} For any $(k_1,\dotsc,k_{r})\in \N^{r}$ and integer $p\geq 0$,
\begin{align}
T(\{1\}_{p-1},\ol{1},k_1,\dotsc,k_{r},\ol{1})=T(\ol{1}\diamond \{1\}_{k_{r}-2}\diamond\ol{1},\dotsc,\ol{1}\diamond \{1\}_{k_{1}-2}\diamond\ol{1},\ol{p+1}),
\end{align}
where $(\{1\}_{-1},1):=\emptyset$.
\end{thm}
\begin{proof}
Setting $k_r=1$ in \eqref{d2} we get
\begin{align*}
\frac{(-1)^{p+r}p!}{2^r}T(\{1\}_{p-1},\ol{1},\bfk_{r-1},\ol{1})=&\int_0^1 \frac{\log^p\left( \frac{1-t}{1+t}\right)\, dt}{1+t^2} \left(\frac{dt}{t}\right)^{k_1-1} \frac{dt}{1+t^2}  \dotsm \left(\frac{dt}{t}\right)^{k_{r-1}-1}\frac{dt}{1+t^2}.
\end{align*}
Then changes of variables $t_j\rightarrow \frac{1-t_{r+1-j}}{1+t_{r+1-j}}$ in the above identity yield
\begin{align*}
&\frac{(-1)^{p+r}p!}{2^r}T(\{1\}_{p-1},\ol{1},\bfk_{r-1},\ol{1})\\
&=2^{k_1+\cdots+k_{r-1}-r+1}\int_0^1 \frac{dt}{1+t^2}\left(\frac{dt}{1-t^2}\right)^{k_{r-1}-1} \dotsm \frac{dt}{1+t^2}\left(\frac{dt}{1-t^2}\right)^{k_{1}-1}  \frac{\log^p(t)\, dt}{1+t^2}\\
&=\frac{(-1)^{p+r}p!}{2^r}T(\ol{1}\diamond \{1\}_{k_{r-1}-2}\diamond\ol{1},\dotsc,\ol{1}\diamond \{1\}_{k_{1}-2}\diamond\ol{ 1},\ol{p+1})
\end{align*}
by \eqref{b3}. We can now complete the proof of the theorem by replacing $r$ by $r+1$.
\end{proof}

In particular, if $k_1=\cdots=k_{r}=1$, then we obtain the well-know result
\[T(\{1\}_{p-1},\ol{1},\{1\}_{r},\ol{1})=T(\{1\}_{r},\ol{p+1}).\]
If letting $r=1,k_1=k$ and $p=0$, then
\begin{align}\label{d8}
T(\bar 1\diamond \{1\}_{k-2}\diamond \bar 1,\bar 1)=T(k,\bar 1).
\end{align}
Hence, if putting $k=2p+1\ (p\geq 0)$ and using \eqref{d6} gives
\begin{align}\label{d9}
T(2p+1,\bar 1)=\frac{(-1)^p}{2}T^2(\ol{p+1})+\sum_{j=1}^p (-1)^{j-1} T(\bar j)T(\ol{2p+2-j}).
\end{align}
On the other hand, from \cite[Corollary~3.4]{WX2020}, we have
\begin{align}
\sum_{n=1}^\infty \frac{h^{(2p)}_n}{n}(-1)^{n-1}=2^{2p-1}T(2p,\bar 1)=p2^{2p+1}t(2p+1)-\sum_{k=1}^p \ol{t}(2k-1)\ol{t}(2p-2k+2).
\end{align}
Then applying \eqref{d8} and the relation $T(\bar k)=-\frac1{2^{k-1}}\ol{t}(k)\ (k\in\N)$, we obtain
\begin{align}
T(\bar 1,\{1\}_{2p-2}, \bar 1,\bar 1)=2pT(2p+1)-\sum_{k=1}^p T(\ol{2k-1})T(\ol{2p-2k+1})\quad (p\in \N).
\end{align}
As an example, setting $p=1$ yields
\begin{align*}
T(\bar 1,\bar 1,\bar 1)=\frac7{2}\z(3)-\frac{\pi}{4}\ol{t}(2).
\end{align*}

\begin{cor} \label{conj:parity}
For any $\bfk=(k_1,\dotsc,k_r)\in \N^r$,
\begin{align}
T(\bfk,\bar 1)+(-1)^{|\bfk|}T(\overleftarrow{\bfk},\bar 1)
=\sum_{l=1}^r (-1)^{|\overleftarrow\bfk_{r-l}|} \sum_{j=1}^{k_l} (-1)^{j-1} T\left(\overleftarrow\bfk_{r-l},\ol{j} \right)T\left(\ora\bfk_{l-1},{\ol{k_l-j+1}} \right).
\end{align}
\end{cor}
\begin{proof} This follows immediately from the Corollary \ref{cord4} and Theorem \ref{thmd5} with $p=0$.
\end{proof}

We end this section by a parity conjecture for AMMVs based on our computations.
\begin{con}\label{MMV-conj} For composition $\bfk=(k_1,k_2,\ldots,k_r)$
and $\bfeps=(\eps_1, \dots, \eps_r)\in\{\pm 1\}^r$ and $\bfsi=(\sigma_1, \dots, \sigma_r)\in\{\pm 1\}^r$ with $(k_r,\sigma_r)\neq (1,1)$,
\begin{align*}
\left\{1-\prod_{j=1}^r \Big((-1)^{k_j+1} \sign(\sigma_j+\eps_j+1)\Big) \right\}M_{\bfsi}(\bfk;\bfeps)
\end{align*}
can be expressed in terms of products of AMMVs of lower depth.
\end{con}
For example, we have the following cases of double $M$-values
\begin{align*}
&M_{1,-1}(\check{2},\check{6})=-\frac{5\pi^5}{1536}\z(3)-\frac{\pi^3}{64}\z(5)-\frac{381\pi}{4096}\z(7)+2G_4,\\
&M_{-1,1}(\check{2},\check{6})=-\frac{\pi^6}{240}G_1+\frac{5\pi^5}{2048}+\frac{15\pi^3}{1024}\z(5)+\frac{381\pi}{4096}\z(7)+2G_4,\\
&M_{-1,1}(\check{2},6)=-\frac{\pi^6}{7680}G_1-\frac{7\pi^4}{960}G_2-\frac{5\pi^2}{12}G_3+\frac{381\pi}{64}\z(7)-14G_4,\\
&M_{-1,1}(\check{6},2)=\frac{35\pi^5}{1536}\z(3)+\frac{31\pi^3}{64}\z(5)-\frac{\pi^2}{4}G_3+\frac{381\pi}{64}\z(7)-42G_4,\\
&M_{-1,-1}(\check{2},6)=\frac{\pi^6}{7680}G_1+\frac{\pi^4}{120}G_2+\frac{4\pi^2}{3}G_3-14G_4,\\
&M_{-1,-1}(\check{6},2)=\frac{\pi^6}{240}G_1+\frac{\pi^4}{8}G_2+\frac{11\pi^2}{4}G_3-42G_4.
\end{align*}
Here
$$
G_m=\sum_{n\ge 0}\frac{(-1)^n}{(2n+1)^{2m}}
$$
is the generalized Catalan's constant. Clearly, $G_m=\ol{t}(2m)/2^{2m}=-T(\ol{2m})/2$. In particular, $G=G_1$ is Catalan's constant.
\begin{re}
It should be emphasized that Panzer proves a general parity result on multiple polylogarithms in \cite[Theorem~1.3]{Panzer2017}. Let $\AMMV_w$ be the $\Q$-vector space generated by all the AMMVs of weight $w$ and denote all its subspaces similarly.
According to the definition of colored MZVs of level four, we know that $\eta_j\in\{e^{\pi i/2 },e^{\pi i},e^{3\pi i/2},e^{2\pi i}\}=\{\pm1,\pm i\}$. Hence, for colored MZVs of level four, we have
\begin{align*}
\eta_j^{n_j}:=
\left\{
  \begin{array}{ll}
\{\pm1,\pm i (-1)^{{(n_j-1)}/{2}}\}\ &\quad \hbox{if $n_j$\ \text{odd};} \\
\{1,\pm (-1)^{{n_j}/{2}}\}\ &\quad \hbox{if $n_j$\ \text{even},}
  \end{array}
\right.
\end{align*}
and
\begin{align*}
Li_{k_1,k_2,\ldots, k_r}(\eta_1,\dots,\eta_r)&=\sum\limits_{0<n_1<\cdots<n_r}
\frac{\eta_1^{n_1}\dots \eta_r^{n_r}}{n_1^{k_1} \dots n_r^{k_r}}\\
&=\sum_{\gd_1=0}^1 \cdots \sum_{\gd_r=0}^1
\sum_{\substack{0<n_1<\dots<n_r \\ n_1\equiv \gd_j \pmod{2} \ \forall j }}\frac{\eta_1^{n_1} \dots \eta_r^{n_r}}{n_1^{k_1} \dots n_r^{k_r}} \in \AMMV \cup i \AMMV.
\end{align*}
On the other hand, from the definition of AMMVs, it is clear that the $\Q$-vector space $\AMMV_w$
lies in $(\CMZV_{4,w}\cup i\CMZV_{4,w})\cap \R$, where
$\CMZV_{4,w}$ is the $\Q$-vector space generated by all colored MZVs of level four and weight $w$.
Thus, we can obtain a certain parity result on AMMVs from the parity principle for multiple polylogarithms of Panzer.
However, for any $m,n\in\N$ with $m+n=8$
$$
M_{-1,1}(\check{m},\check{n})=\upi \Big( Li_{m,n}(\upi,1)-Li_{m,n}(-\upi,1)-Li_{m,n}(\upi,-1)+Li_{m,n}(-\upi,-1)\Big).
$$
By Deligne's result \cite[Theorem~6.1]{Deligne2010} every CMVZ of depth 2 and weight $8$ is a $\Q$-linear combinations of double logarithm values $Li_{m,n}(\upi,1)$ and products of $(2\pi \upi)^p Li_{8-p}(\upi)$ ($0\le p\le 8$). On the other hand, it is generally believed that not all $Li_{m,n}(\upi,1)$ with $m+n=8$ can be reduced in the above sense (cf. \cite{Broadhurst1996}). Thus, our Conjecture~\ref{conj:parity} should not be a consequence of Panzer's parity principle for multiple polylogarithms.
\end{re}

\section{Dimension Computation of AMTVs}
We first relate AMTVs to colored MZVs via the following mechanism. For all $\bfk\in\N^r$,
$\bfeta\in\{\pm1\}^r$, if $(k_r,\eta_r)\ne(1,1)$ then
\begin{align*}
T(\bfk;\bfeta)
=&\, 2^r\cdot \sum_{0<n_1<\cdots<n_r} \frac{ \eta_1^{n_1}\eta_2^{n_2}\dots \eta_r^{n_r}}{(2n_1-1)^{k_1}(2n_2-2)^{k_2}\dotsm (2n_r-r)^{k_r}} \\
=&\, 2^r\cdot \sum_{0<n_1<\cdots<n_r} \frac{\xi_1^{2n_1}\xi_2^{2n_2} \dots \xi_r^{2n_r}}{(2n_1-1)^{k_1}(2n_2-2)^{k_2}\dotsm (2n_r-r)^{k_r}}\\
=&\, \xi_1\xi_2^2\dots \xi_r^r \cdot \sum_{0<n_1<\cdots<n_r} \frac{ \xi_1^{m_1}(1-(-1)^{m_1})\xi_2^{m_2}(1+(-1)^{m_2})\dotsm \xi_r^{m_r}(1+(-1)^{r+m_r})}{m_1^{k_1}m_2^{k_2}\cdots m_r^{k_r}},
\end{align*}
where $\xi_j=1$ if $\eta_j=1$ and $\xi_j=\sqrt{-1}$ if $\eta_j=-1$. Thus the $\Q$-vector space $\AMTV_w$
generated by all AMTVs of weight $w$ lies in $(\CMZV_{4,w}\cup i\CMZV_{4,w})\cap \R$.

We remark first that $\CMZV_{4,w}\cap \R$ does not contain $\AMTV_w$ in general. In fact,
in weight one we have only one AMTV, i.e.,
$$
T(\bar 1)=2\sum_{n=1}^\infty \frac{(-1)^n}{2n-1}=-\frac{\pi}{2}.
$$
On the other hand,
$\CMZV_{4,1}$ is generated by the three colored MZVs
\begin{equation*}
Li_1(i)=-\frac12\log 2+\frac\pi4 i, \quad Li_1(-i)=-\frac12\log 2-\frac\pi4 i, \quad Li_1(-1)=-\log 2.
\end{equation*}
Therefore, $\CMZV_{4,1}\cap \R= (\log 2)\Q .$ But it is well-known that $\log 2$ and $\pi$ are linearly independent over $\Q$. Indeed, if $\log 2=r\pi$ for some nonzero $r\in\Q$, then
\begin{equation*}
2=e^{r\pi}=(-1)^{ri}\in\Q.
\end{equation*}
However, $(-1)^{ri}$ is transcendental by Gelfond--Schneider theorem \cite{Gelfond1934} since $ri$ is algebraic but not rational. This contradiction implies that  $\AMTV_1\not\subseteq \CMZV_{4,1}\cap\R.$

Set $Li_\emptyset=1$. Deligne shows in \cite[Theorem~6.1]{Deligne2010} that $\CMZV_{4,w}$ can be generated by
\begin{equation}\label{equ:CMZVbasis}
\bfB_w:=\Big\{ (2\pi i)^p  Li_\bfk(i,1,\dotsc,1) : p+|\bfk|=w\Big\},
\end{equation}
where $p\in\N_0$ and $\bfk$ runs through compositions of positive integers or the empty set.
For example, setting $L_\bfk= Li_\bfk(i,\{1\}_{d-1})$ for all $\bfk\in\N^d$, we have
\begin{align*}
 \bfB_1:=&\left\{  2\pi i, L_1=-\frac12\log 2+\frac{\pi}{2}i\right\} \quad\text{or}\quad
 \bfB'_1:=\Big\{   \pi i, \log 2 \Big\};  \\
 \bfB_2:=&\Big\{  (2\pi i)^2, (2\pi i)L_1, L_2, L_{1,1}\Big\}  \quad\text{or}\quad
 \bfB'_2:=\Big\{   \pi^2, (\pi \log 2)i, G_1, \log^2(2)  \Big\}; \\
 \bfB_3:=&\Big\{  (2\pi i)^3,  (2\pi i)^2 L_1, (2\pi i)L_2, (2\pi i) L_{1,1},
            L_3, L_{1,1,1}, L_{2,1}, L_{1,2}\Big\} \\
 \text{or }\quad  \bfB'_3:=&\Big\{  \pi^3 i,  \pi^2 \log 2, \pi G_1, (\pi \log^2 2 )i ,
            \zeta(3), \log^3 2, L_{2,1}, L_{1,2}\Big\}.
\end{align*}
Similarly, denote by
$\bfTB_w$ the conjectural basis of $\Q$-vector space generated by all AMTVs of weight $w$.

\begin{thm}\label{thm-Dim}
For $w\le 4$, the follow set $\bfTB_w$ and $\bfTB_w(\Z)$ generates the space $\AMTV_w$:
\begin{align*}
& \bfTB_1=\bfTB_1(\Z)=\big\{T(\bar1)\big\},\\
& \bfTB_2=\bfTB_2(\Z)=\big\{T(1,\bar 1),T(\bar2)\big\}, \\
&\bfTB_3=\bfTB_3(\Z)=\big\{T(3), T(1,\bar 2), T(\bar2, \bar1), T(1,1,\bar1) \big\},\\
 &\bfTB_4=\big\{T(1,\bar3),T(2,\bar2),T(1,1,\bar2),T(\bar3,\bar1),T(1,\bar2,\bar1),T(\bar2,1,\bar1),T(1,1,1,\bar1)
 \big\},\\
 &\bfTB_4(\Z)=\big\{T(1,1,1,\bar1),\ T(\bar2,2),\ T(\bar1,3),\ T(1,\bar3),\ T(\bar2,1,\bar1),\ T(3,\bar1),\ T(\bar4) \big\},\\
 &\bfTB_5=\left\{
\begin{aligned}
& T(2,\bar3),T(\bar3,\bar2),T(1,1,\bar3),T(1,2,\bar2),T(\bar2,1,\bar2),T(1,1,1,\bar2),T(1,\bar3,\bar1), \\
& T(\bar3,1,\bar1),T(\bar2,\bar2,\bar1),T(1,1,\bar2,\bar1),T(1,\bar2, 1,\bar1),T(\bar2,1,1,\bar1),T(1,1,1,1,\bar1)
\end{aligned}
\right\},\\
 &\bfTB_5(\Z)=\left\{
\begin{aligned}
& T(\bar5), T(5), T(\bar4, \bar1), T(\bar3, \bar2), T(\bar3, 2), T(1, \bar4), T(1, 4), T(3, \bar2), \\
&  T(\bar3, \bar1, \bar1), T(\bar2, \bar2, \bar1), T(\bar2, \bar1, 2), T(\bar2, 1, 2), T(\bar1, 1, 3)
\end{aligned}
\right\}.
\end{align*}
If a variation of Grothendieck's period conjecture \emph{\cite[Conj.~5.6]{Deligne2010}} holds then $\bfTB_w$
is a basis of $\AMTV_w$ over $\Q$.
Furthermore, the bases $\bfTB_w(\Z)$ above are all integral basis, namely, every AMTV of weight $w$ is a $\Z$-linear
combination of the basis elements in $\bfTB_w(\Z)$ for all $w\le 5$.
\end{thm}
\begin{proof}
For an admissible composition $\bfk=(k_1,\dotsc,k_r;\sigma_1,\dotsc,\sigma_r)$ of weight $w$
we can set
\begin{equation*}
T(\bfk)=\sigma_1\sigma_2^2\dots \sigma_r^r \int_0^1 \bfp(\bfk):=\sigma_1\sigma_2^2\dots \sigma_r^r \int_0^1 \om_{i_1}\dots \om_{i_w}.
\end{equation*}
by \eqref{equ:bfp}. Define
\begin{equation*}
\alpha(\bfk):=\sharp\{j: i_j=0\},\quad \beta(\bfk):=\sharp\{j: i_j=1\},\quad \gamma(\bfk):=\sharp\{j: i_j=-1\}.
\end{equation*}
Then from the duality relation \eqref{equ:duality} we see that
\begin{equation*}
\alpha(\bfk)=\beta(\bfk^*), \quad\beta(\bfk)=\alpha(\bfk^*),\quad\gamma(\bfk)=\gamma(\bfk^*).
\end{equation*}
Therefore we have
\begin{equation*}
\alpha(\bfk)+\alpha(\bfk^*)=w-\gamma(\bfk)=w-\gamma(\bfk^*).
\end{equation*}
So to compute the vector space $\AMTV_w$ we only need to consider AMTVs $T(\bfk)$ with
\begin{equation*}
     2\alpha(\bfk)\le w-\gamma(\bfk).
\end{equation*}
Furthermore, if the equality holds in the above, then from the two dual values we only
choose the one so that the first $\om_0$ in its iterated integral form appears earlier.
For example, by \eqref{equ:duality}
\begin{align*}
 & T(\bar 1,\bar 1)=-\int_0^1 \om_1\om_{-1}=-\int_0^1 \om_{-1}\om_0= T(\bar 2),\ T(1,2)=T(3),\\
 & T(1,\bar 1,\bar 1)=T(\bar 3),\ T(\bar 1,1,\bar 1)=T(1,\bar 2),\
  T(\{\bar1\}_3)=T(2,\bar 1),\ T(\bar 1,2)=T(\bar 2,\bar 1) .
\end{align*}
Clearly, there are exactly $4\cdot 3^{w-2}$ weight $w$ AMTVs if we consider their iterated integral representations. Furthermore, the self-dual values are determined by the first half 1-forms in their
iterated integral representations (and the middle 1-form must be $\om_{-1}$ if the weight $w$ is odd). So there are
$2\cdot 3^{[w/2-1]}$ such values. Therefore, by duality
we only need to compute $N_w:=2\cdot 3^{w-2}+3^{[w/2-1]}$ values
in $\AMTV_w$ to determine its dimension. We have $N_2:=3$, $N_3:=7$, $N_4:=21$, $N_5:=57$.

Now, straight-forward computation leads quickly to the conclusions in weight $w\le 3$:
\begin{align*}
& T(\bar 1)=-\frac{\pi}{2};
\qquad  T(2)=2T(1,\bar 1),\ T(\bar 2)=-2G_1,\ T(1,\bar 1)=\frac{\pi^2}{8}; \\
& T(\bar3)=3T(1,1,\bar1),T(\bar1,\bar2)=6T(1,1,\bar1)-2T(\bar2,\bar1), T(2,\bar1)=T(3)-T(1,\bar2).
\end{align*}
By duality, $T(\bar1,\bar1)$ and the other five weight 3 AMTVs can be found easily.
We see immediately that the spans of $\bfTB_2$ and $\bfTB_3$ in the theorem
generate the space $\AMTV_2$ and $\AMTV_3$ over $\Q$, respectively. Moreover,
\begin{equation*}
 \AMTV_2\otimes \Q(i)=\big \langle \tpi^2, L_2 \big \rangle_{\Q(i)}\quad \text{ and } \quad
 \AMTV_3\otimes \Q(i)=\big \langle \tpi^3, \tpi L_2, L_{2,1}, L_3 \big\rangle_{\Q(i)},
\end{equation*}
where $\tpi=2\pi i$ and $\quad L_\bfk=Li_\bfk(i,1,\dotsc,1).$ So the bases $\bfTB_2$ and $\bfTB_3$
are linearly independent over $\Q$ if we assume \cite[Conj.~5.6]{Deligne2010} which implies that
Deligne's basis \eqref{equ:CMZVbasis} is linearly independent over $\Q(i)$ for all for $w$.

In weight 4, using the computation from \cite{Zhao2010b}, we find by Maple that $\AMTV_4$ can be
spanned over $\Q(i)$ by the seven elements from Deligne's basis \eqref{equ:CMZVbasis}:
\begin{equation*}
 \Big\{\tpi^4, \tpi^2 L_2, \tpi L_3, \tpi L_{2,1}, L_4, L_{3,1}, L_{2,2} \Big\}
\end{equation*}
which is conjecturally linearly independent over $\Q(i)$. Further, it is straight-forward to change this basis
to the basis $\bfTB_4$ in the theorem. To save space, we set
\begin{alignat*}{7}
 & a_1=T(1,1,1,\bar1),\quad && a_2=T(\bar2,2),\quad && a_3=T(\bar1,3),\quad && a_4=T(1,\bar3), \\
 & a_5=T(\bar2,1,\bar1),\ && a_6=T(3,\bar1),\ && a_7=T(\bar4). && \
\end{alignat*}
Then
\begin{alignat*}{5}
& T(4)=8a_1, & &
T(\bar2,\bar2)=9a_2-12a_7+18a_3,& &
T(\bar1,\bar2,\bar1)=24a_1-4a_5-2a_6,  \\
& T(2,2)=4a_4,& &
T(\bar1,\bar3)=3a_7-4a_3-2a_2,& &
T(1,1,\bar2)=3a_2-4a_7+6a_3,\\
& T(\bar2,\bar1,\bar1)=2a_5,& &
T(2,\bar2)=12a_1-2a_4-a_6,& &
T(\bar3,\bar1)=9a_7-12a_3-6a_2, \\
& T(1,3)=6a_1-2a_4, \quad & &
T(\bar1,1,\bar2)=12a_1-2a_6, & &
T(2,1,\bar1)=2a_2+6a_3-3a_7, \\
& & &
T(1,2,\bar1)=-4a_2-9a_3+6a_7, \quad & &
T(\bar1,2,\bar1)=-12a_1+a_5+2a_6.
\end{alignat*}
All the other 15 weight 4 AMTVs can be obtained by duality Theorem~\ref{thm:dualityAMTV}, for e.g.:
\begin{equation*}
T(1,\bar1,2)=a_5,\ T(\bar1,\bar1,\bar2)=T(\bar1,2,\bar1),\
T(\bar 1,1,1,\bar 1)=T(1,1,\bar 2),\  T(\bar 1,1,\bar 1,\bar 1)=T(3,\bar 1).
\end{equation*}
The last two equations are consistent with \eqref{c21} and \eqref{d8}, too.

For weight 5, similar computation can be carried out using Au's Mathematica package
\cite[Appedix A]{Au2020} containing the explicit
expressions of all values in $\CMZV_{4,5}$ in terms of 32 basis elements. It can be further found
with Maple that $\AMTV_5$ is contained in a co-dimension one $\Q(i)$-subspace of the space generated
by the following 14 elements in the Deligne basis:
\begin{equation*}
\tpi^5, \tpi L_{2,2}, \tpi L_{3,1}, \tpi L_{4}, \tpi^2L_{1,2}, \tpi^2L_{3},
\tpi^3L_{2},  L_{5}, L_{3,2}, L_{4,1},L_{2,1,2}, L_{2,2,1}, L_{2,3}, L_{3,1,1}.
\end{equation*}
Finally, it can be verified that $\AMTV_5$ can be generated by the 13 elements in $\bfTB_5$.
To save space, set
\begin{align*}
& b_{1}:=T(1,1,1,1,\bar1),\ b_{2}:=T(2,1,1,\bar1),\ b_{3}:=T(1,1,2,\bar1),\ b_{4}:=T(\bar1,2,1,\bar1),\\
& b_{5}:=T(\bar2,1,1,\bar1),\ b_{6}:=T(1,3,\bar1),\ b_{7}:=T(1,\bar3,\bar1),\ b_{8}:=T(3,\bar1,\bar1),\\
& b_{9}:=T(\bar3,\bar1,\bar1),\ b_{10}:=T(\bar2,\bar2,\bar1),\ b_{11}:=T(\bar1,4),\ b_{12}:=T(1,\bar4),\ b_{13}:=T(1,4).
\end{align*}
Then
\begin{alignat*}{3}
& T(\bar5)=25b_{1},  & &
T(\bar3,2)=2b_{6}+4b_{7}-77b_{11}+240b_{1}+21b_{4}-17b_{5}+2b_{8}, \\
& T(5)=2b_{2}-2b_{12}+7b_{13},   & &
T(\bar2,3)=-90b_{1}-10b_{4}+8b_{5}+35b_{11}-b_{6}-2b_{7}-b_{8},  \\
&  T(2,3)=2b_{13}+2b_{2}-2b_{12},   & &
T(\bar1,\bar2,\bar2)=-b_{10}-6b_{2}+12b_{12}-24b_{13}+3b_{3}+4b_{9}, \\
& T(2,\bar3)=-4b_{12}+4b_{13}+b_{2}, & &
T(2,\bar1,\bar2)=180b_{1}-70b_{11}+20b_{4}-10b_{5}+2b_{6}+4b_{7}+2b_{8}, \\
& T(3,2)=6b_{13}, & &
T(1,\bar1,3)=9b_{12}+3b_{13}-6b_{3}-2b_{9}+3b_{2},     \\
& T(\bar1,\bar4)=-20b_{1}+6b_{11}-2b_{4}+2b_{5},   & &
T(\bar1,2,\bar2)=2b_{10}+3b_{2}+12b_{13}-6b_{3}-3b_{9}, \\
& T(\bar3,\bar2)=2b_{4}-2b_{5}+60b_{1}-14b_{11},   & &
T(\bar1,3,\bar1)=-2b_{10}-2b_{2}-6b_{13}+5b_{3}+2b_{9}, \\
& T(\bar2,\bar3)=60b_{1}-14b_{11}+5b_{4}-5b_{5},    & &
 T(\bar2,\bar1,2)=4b_{5}-120b_{1}+28b_{11}-8b_{4}+2b_{8},  \\
& T(3,\bar2)=12b_{12}-6b_{13}-3b_{3}, & &
T(2,\bar2,\bar1)=-2b_{7}-60b_{1}+4b_{5}+28b_{11}-8b_{4}-b_{8},  \\
& T(\bar4,\bar1)=21b_{11}-60b_{1}-5b_{4}+5b_{5},   & &
T(\bar1,\bar3,\bar1)=b_{10}+2b_{2}-12b_{12}+12b_{13}+b_{3}-b_{9}, \\
& T(4,\bar1)=2b_{2}-12b_{12}+12b_{13}+3b_{3},  & &
T(2,1,\bar2)=-60b_{1}+42b_{11}-12b_{4}+6b_{5}-2b_{6}-4b_{7}-3b_{8}, \\
& T(\bar3,1,\bar1)=3b_{12}, & &
T(2,2,\bar1)=-2b_{7}-60b_{1}+42b_{11}-12b_{4}+6b_{5}-4b_{6}-3b_{8}, \\
& T(\bar1,\bar1,\bar3)=-24b_{12}+12b_{13}+9b_{3}+b_{9},  & &
T(3,1,\bar1)=b_{7}+30b_{1}-14b_{11}+4b_{4}-2b_{5}+b_{6}+b_{8}, \\
& T(\bar1,1,\bar3)=9b_{12}-6b_{13}-3b_{3},  & &
T(1,1,\bar3)=-b_{7}-20b_{1}+14b_{11}-4b_{4}+2b_{5}-b_{6}-b_{8}, \\
& T(1,\bar2,2)=b_{9}-3b_{2}+6b_{12}-6b_{13},   & &
T(1,2,\bar2)=4b_{7}+90b_{1}-56b_{11}+16b_{4}-8b_{5}+3b_{6}+4b_{8}, \\
& T(1,\bar1,\bar3)=4b_{5}-90b_{1}+28b_{11}-8b_{4},   & &
T(\bar2,\bar1,\bar2)=-b_{10}+6b_{2}+12b_{12}+12b_{13}-9b_{3}-4b_{9}, \\
& T(1,\bar2,\bar2)=30b_{1}-b_{6}-2b_{7}-b_{8},& &
   T(\bar2,2,\bar1)=b_{10}+3b_{2}-12b_{12}+12b_{13}-b_{9}, \\
& T(\bar2,1,\bar2)=12b_{12}-b_{10}-6b_{13}-3b_{3},    & &
T(\bar1,1,1,\bar2)=-100b_{1}+14b_{11}-2b_{4}+6b_{5}+2b_{8},\\
& T(2,\bar1,2)=-12b_{12}+6b_{3}+2b_{9},  & &
T(\bar1,1,2,\bar1)=180b_{1}-35b_{11}+7b_{4}-11b_{5}-2b_{8}, \\
&  T(\bar1,\bar2,1,\bar1)=60b_{1}-14b_{11}+3b_{4}-5b_{5},  \quad & &
T(\bar1,\bar1,2,\bar1)=-360b_{1}+84b_{11}-20b_{4}+24b_{5}+2b_{8}, \\
& T(1,1,1,\bar2)=2b_{12}-b_{13}-b_{3},   & &
T(\bar2,\bar1,1,\bar1)=150b_{1}+13b_{4}-10b_{5}-49b_{11}+b_{6}+2b_{7}+b_{8}, \\
& T(1,2,1,\bar1)=-3b_{12}+3b_{13},   & &
T(2,\bar1,\bar1,\bar1)=36b_{5}-480b_{1}-44b_{4}+168b_{11}-4b_{6}-8b_{7}-4b_{8}.
\end{alignat*}
All the other 51 weight 5 AMTVs can be obtained by duality Theorem~\ref{thm:dualityAMTV}, for e.g.,
\begin{align*}
&T(1,\bar1,\bar1,\bar2)=b_4,\
T(1,1,\bar1,2)=b_5,\
T(\bar1,1,3)=b_7,\
T(\bar2,1,2)=b_8,\\
&T(\bar1,\bar1,3)=T(1,\bar2,\bar2),\
T(\bar1,\bar2,2)=T(2,\bar1,\bar2),\
T(1,\bar2,\bar1,\bar1)=T(\bar2,\bar1,1,\bar1),\\
&T(1,\bar1,\bar2,\bar1)=T(\bar1,\bar2,1,\bar1),\
T(\bar1,\bar1,1,\bar2)=T(\bar1,1,2,\bar1),\
T(\bar1,2,2)=T(2,\bar2,\bar1).
\end{align*}

This completes the proof of Theorem \ref{thm-Dim}.
\end{proof}

We remark that the integral structure in Theorem~\ref{thm-Dim} is unlikely to hold in all higher weights. For the corresponding structure for MZVs, we know this integrality breaks down at weight 6 and 7 (see \cite[Sec. 2]{ConradZ2009}).

We end the paper by the following conjecture on the structure of AMVTs.
\begin{con} \label{conj:tribonacci2}
Let $\AMTV_n$ be the $\Q$-vector space generated by all AMVTs of weight $n$. Set $\AMTV_0=1$.
Then
\begin{equation*}
    \sum_{n=0}^\infty (\dim_\Q \AMTV_n)t^n = \frac{1}{1-t-t^2-t^3}.
\end{equation*}
Namely, the dimensions form the tribonacci sequence $\{d_w\}_{w\ge 1}=\{1,2,4,7,13,24,\dotsc\}$, see A000073 at oeis.org.
\end{con}

We verified the conjecture rigorously when the weight is less than 6 assuming Grothendieck's period conjecture.
In weight 6, we verified $\dim_\Q \AMTV_6=24$ numerically by using PSLQ \cite{BaileyBr2001}. By duality relations in Theorem~\ref{thm:dualityAMTV}
we only need to express 160 AMTVs in terms of the Deligne basis \eqref{equ:CMZVbasis}. Since we
do not have double shuffle structure, to cut down computation time
we can produce some additional linear relations such as the
lifted relations. We can obtain these by multiplying every AMTV of weight $k$ ($k=2,3$)
on each duality relations in weight $6-k$. After this step, we reduce the number of AMTVs to be computed to 136.
We can further reduce this number to 131 if we use the following result: for all $(k_j,\sigma_j)\ne(1,1)$
\begin{align*}
&\sigma_3T(k_1,k_2,k_3;\sigma_1,\sigma_2,\sigma_3)-T(k_1,k_2;\sigma_1,\sigma_2)T(k_3;\sigma_3)\\
=&\sigma_1T(k_3,k_2,k_1;\sigma_3,\sigma_2,\sigma_1)-T(k_3,k_2;\sigma_3,\sigma_2)T(k_1;\sigma_1),
\end{align*}
which follows immediately from the \cite[(4.5)]{WWXC2020}.

Moreover, for all $n\le 6$ a slight modification of the following set
of AMTVs form a basis of $\AMTV_n$:
\begin{equation*}
\bfTB'_n=\left\{ T(s_1,\dotsc,s_m;\eps_1,\dotsc,\eps_m) \left|
\begin{aligned}
& s_1,\dotsc,s_m\in\{1,2,3\}; \ \eps_m=-1; \\
&\eps_j=\sign(1.5-s_j)\ \forall j<m
\end{aligned}
\right.\right\}.
\end{equation*}
From Theorem~\ref{thm-Dim} we see that $\bfTB_1=\bfTB'_1$, $\bfTB_2=\bfTB'_2$ but we
have to make the following adjustments for larger weights:
\begin{align*}
  &   \bfTB_3: \text{ replace $T(\bar3)$  by  $T(3)$};  \qquad
     \bfTB_4: \text{ replace $T(\bar2,\bar2)$  by  $T(2,\bar2)$};   \\
  &   \bfTB_5: \text{ replace $T(\bar2,\bar3)$  by  $T(2,\bar3)$, and $T(1,\bar2,\bar2)$  by  $T(1,2,\bar2)$};  \\
  &   \bfTB_6: \text{ replace $T(\bar3,\bar3)$  by  $T(3,\bar3)$}.
\end{align*}
So far, we have not been able to find an apparent pattern of the bases which holds for all the above cases.

\medskip
\noindent
{\bf Acknowledgments.}  The authors expresses their deep gratitude to Professors Masanobu Kaneko and Weiping Wang for valuable discussions and comments. The first author is supported by the Scientific Research
Foundation for Scholars of Anhui Normal University.

 \end{document}